\documentclass[dvipsnames,11pt,reqno]{amsart}

  \usepackage{amssymb,mathtools,amsthm}
  \usepackage[a4paper,left=3cm,right=3cm,top=3cm,bottom=3cm]{geometry}
  \usepackage[T1]{fontenc}
  \usepackage[utf8]{inputenc}
\usepackage{microtype}
\usepackage{subcaption}

\newcommand\numberthis{\addtocounter{equation}{1}\tag{\theequation}}
\newcounter{n}
\numberwithin{n}{section}
\theoremstyle{plain}
  \newtheorem{lemma}[n]{Lemma}
  \newtheorem{theorem}[n]{Theorem}
  \newtheorem{proposition}[n]{Proposition}
  \newtheorem{corollary}[n]{Corollary}
\theoremstyle{definition}
  \newtheorem{definition}[n]{Definition}

  \usepackage{tikz}
  \usetikzlibrary{patterns}
  \usetikzlibrary{decorations.pathreplacing}

  \usepackage{enumitem}

\usepackage{hyperref}
\definecolor{colorlinks}{RGB}{0, 24, 168}
\definecolor{colorcites}{RGB}{124, 10, 2}
\hypersetup{
    colorlinks=true,
    linkcolor=colorlinks,
    citecolor=colorcites,
    urlcolor=colorlinks,
    pdfborder={0 0 0}
}

  \usepackage{setspace}

\usepackage{todonotes}

\renewcommand{\t}{\mathbf{s}}
\renewcommand{\a}{\mathbf{a}}
\renewcommand{\b}{\mathbf{b}}

\newcommand{\e}{\mathbf{e}}

\newcommand{\n}{\mathbf{n}}
\newcommand{\p}{\mathbf{p}}
\newcommand{\f}{\mathbf{f}}
\newcommand{\x}{\mathbf{x}}
\renewcommand{\u}{\mathbf{u}}
\newcommand{\z}{\mathbf{z}}

\newcommand{\s}{\mathbf{s}}
\newcommand{\m}{\mathbf{m}}
\newcommand{\y}{\mathbf{y}}
\newcommand{\0}{\mathbf{0}}

\newcommand{\R}{\mathbb R}

\newcommand{\T}{\mathbb T}
\renewcommand{\S}{\mathbb S}
\newcommand{\N}{\mathbb N}
\newcommand{\V}{\mathbb V}
\newcommand{\E}{\mathbb E}
\renewcommand{\P}{\mathbb P}
\newcommand{\Z}{\mathbb Z}

\newcommand\torusedges{\mathbb E(\mathbb V)}

\newcommand{\Var}{\operatorname{Var}}

  \renewcommand{\gcd}{\operatorname{gcd}}

\newcommand\optionalindent{}

\begin{document}

  % HEADER
  \title[Walk on fractures of a hypertorus]{Diffusivity of a walk on fractures of a hypertorus}
  \makeatletter
  \@namedef{subjclassname@2020}{\textup{2020} Mathematics Subject Classification}
  \makeatother
  \author{Piet Lammers}
  \date{}
  \address{Statistical Laboratory, Centre for Mathematical Sciences, University of Cambridge}
  \email{p.g.lammers@statslab.cam.ac.uk}
  % !TEX root = ../ms.tex

\begin{abstract}
This article studies discrete height functions on the discrete hypertorus. These are functions on the vertices of this hypertorus graph for which the derivative satisfies a specific condition on each edge. We then perform a random walk on the set of such height functions, in the spirit of \emph{Diffusivity of a random walk on random walks}, a work of Boissard, Cohen, Espinasse, and Norris. The goal is to estimate the diffusivity of this random walk in the mesh limit. It turns out that each height functions is characterised by a number of so-called \emph{fractures} of the hypertorus. These fractures are then studied in isolation; we are able to understand their asymptotic behaviour in the mesh limit due to the recent understanding of the associated random surfaces. This allows for an asymptotic reduction to a one-dimensional continuous system consisting of $\operatorname{gcd}\mathbf n$ parts where $\mathbf n\in\mathbb N^d$ is the fundamental parameter of the original model. We then prove that the diffusivity of the random walk tends to $1/(1+2\operatorname{gcd}\mathbf n)$ in this mesh limit.
\\

\noindent{\scshape R\'esum\'e.}
Cet article étudie des fonctions de hauteur discrètes sur l'hypertore discret.
Il s'agit de fonctions définies sur les sommets de l'hypertore dont la dérivée satisfait une certaine condition en chaque arête.
Nous considérons une marche aléatoire sur cet ensemble de fonctions, à l'instar des travaux de Boissard, Cohen, Espinasse et Norris dans leur article \emph{Diffusivity of a random walk on random walks}.
L'objectif est d'estimer la diffusivité de cette marche aléatoire dans la limite d'échelle.
Nous montrons que toute fonction de hauteur est caractérisée par le nombre de \emph{fractures} qu'elle induit sur l'hypertore.
Nous étudions ensuite ces fractures ; il est possible de comprendre leur comportement asymptotique dans la limite d'échelle grâce à de récents travaux sur les surfaces aléatoires qui leur sont associées.
Cela permet de réduire notre étude asymptotique à un système continu à une dimension, constitué de $\operatorname{pgcd}\n$ parties, où $\n \in \mathbb N^d$ est un paramètre fondamental du modèle initial.
Nous montrons alors que la diffusivité de la marche aléatoire converge vers $1/(1 + 2 \operatorname{pgcd} \n)$ dans cette limite d'échelle.
\end{abstract}
    
  \maketitle
  \setcounter{tocdepth}{1}
  \tableofcontents

  % CONTENTS
  \renewcommand\optionalindent{\tabto{1.2cm}}
  % !TEX root = ../ms.tex
\section{Introduction and main result}

\subsection{Background}
Consider a finite connected bipartite graph $G=(V,E)$,
and consider the set of functions $f:V\to\Z$
which have the property that $|f(y)-f(x)|=1$
for any $\{x,y\}\in E$.
Such functions are called \emph{height functions},
and two height functions $f$ and $g$
are called \emph{neighbours}
whenever $|f-g|$ is identically equal to $1$.
One can now consider the random walk $X$ on the locally finite graph of height functions,
that is, the random walk which at each time increment moves to a neighbour
of the current height function uniformly at random.
Let $r\in V$ denote some distinguished root vertex.
By standard arguments, as $n\to\infty$, the law of
\begin{equation*}\textstyle
  \left(\frac1{\sqrt n}X_{\lfloor nt\rfloor}(r)\right)_{t\in [0,1]}
\end{equation*}
tends to that of a Brownian motion of
some diffusivity $\alpha(G)$ which depends on the graph $G$
only.
For example, it is easy to work out that $\alpha(G)=1$
when $G$ consists of a single vertex,
and that $\alpha(G)=2/3$ when $G$ consists of a single edge.
The problem of determining $\alpha(G)$ is motivated by physics,
where one is interested in the diffusivity of a molecule
as a function of the internal flexibility of that molecule.
See for example~\cite{MR2338265} for a study of \emph{molecular spiders}.

It was proven by Boissard, Cohen, Espinasse, and Norris that $\alpha(G)=2/(n+1)$
whenever $G$ is a line graph consisting of $n$ vertices~\cite{BCEN}.
Espinasse, Guillotin-Plantard, and Nadeau take a different setup,
where the walk is furthermore restricted to height
functions $f$ for which $f(n)-f(1)$ is equal to some fixed value,
where $1$ and $n$ denote the first and last vertex in the line graph
$G$ respectively.
They are able compute the value of $\alpha$
explicitly for this walk~\cite{EGN}.
By fixing the value of $f(n)-f(1)$,
one effectively reduces to working on the circle graph,
or, equivalently, to working with periodic height functions.
The circle graph is in some sense one-dimensional.
The motivating idea behind this article is to replace
the circle graph by a discretisation of the hypertorus,
in any dimension $d\geq 2$.

\begin{figure}
    \begin{center}
      \begin{subfigure}{.35\textwidth}
        \centering
        \begin{tikzpicture}
          \begin{scope}
\clip (-0.375, -1.125) -- (2.625, -1.125) -- (2.625, 3.375) -- (-0.375, 3.375) -- cycle;
\draw[very thick] (-0.375, 0.0) -- (2.625, 0.0);
\draw[very thick] (0.0, -0.375) -- (0.0, 2.625);
\draw[very thick] (-0.375, 0.75) -- (2.625, 0.75);
\draw[very thick] (0.75, -0.375) -- (0.75, 2.625);
\draw[very thick] (-0.375, 1.5) -- (2.625, 1.5);
\draw[very thick] (1.5, -0.375) -- (1.5, 2.625);
\draw[very thick] (-0.375, 2.25) -- (2.625, 2.25);
\draw[very thick] (2.25, -0.375) -- (2.25, 2.625);
\draw[black, fill=white] (0.0, 0.0) circle (0.3);
\node at (0.0, 0.0) {$2$};
\draw[black, fill=white] (0.0, 0.75) circle (0.3);
\node at (0.0, 0.75) {$2.5$};
\draw[black, fill=white] (0.0, 1.5) circle (0.3);
\node at (0.0, 1.5) {$1$};
\draw[black, fill=white] (0.0, 2.25) circle (0.3);
\node at (0.0, 2.25) {$1.5$};
\draw[black, fill=white] (0.75, 0.0) circle (0.3);
\node at (0.75, 0.0) {$2.5$};
\draw[black, fill=white] (0.75, 0.75) circle (0.3);
\node at (0.75, 0.75) {$3$};
\draw[black, fill=white] (0.75, 1.5) circle (0.3);
\node at (0.75, 1.5) {$1.5$};
\draw[black, fill=white] (0.75, 2.25) circle (0.3);
\node at (0.75, 2.25) {$2$};
\draw[black, fill=white] (1.5, 0.0) circle (0.3);
\node at (1.5, 0.0) {$1$};
\draw[black, fill=white] (1.5, 0.75) circle (0.3);
\node at (1.5, 0.75) {$1.5$};
\draw[black, fill=white] (1.5, 1.5) circle (0.3);
\node at (1.5, 1.5) {$2$};
\draw[black, fill=white] (1.5, 2.25) circle (0.3);
\node at (1.5, 2.25) {$0.5$};
\draw[black, fill=white] (2.25, 0.0) circle (0.3);
\node at (2.25, 0.0) {$1.5$};
\draw[black, fill=white] (2.25, 0.75) circle (0.3);
\node at (2.25, 0.75) {$2$};
\draw[black, fill=white] (2.25, 1.5) circle (0.3);
\node at (2.25, 1.5) {$0.5$};
\draw[black, fill=white] (2.25, 2.25) circle (0.3);
\node at (2.25, 2.25) {$1$};
\end{scope}
        \end{tikzpicture}
        \subcaption{$f$}
        \label{fig_small_example_A}
      \end{subfigure}
      \qquad
      \begin{subfigure}{.35\textwidth}
        \centering
        \begin{tikzpicture}
          \begin{scope}
\clip (-1.125, -1.125) -- (3.375, -1.125) -- (3.375, 3.375) -- (-1.125, 3.375) -- cycle;
\draw[very thick] (-1.125, -0.75) -- (3.375, -0.75);
\draw[very thick] (-0.75, -1.125) -- (-0.75, 3.375);
\draw[very thick] (-1.125, 0.0) -- (3.375, 0.0);
\draw[very thick] (0.0, -1.125) -- (0.0, 3.375);
\draw[very thick] (-1.125, 0.75) -- (3.375, 0.75);
\draw[very thick] (0.75, -1.125) -- (0.75, 3.375);
\draw[very thick] (-1.125, 1.5) -- (3.375, 1.5);
\draw[very thick] (1.5, -1.125) -- (1.5, 3.375);
\draw[very thick] (-1.125, 2.25) -- (3.375, 2.25);
\draw[very thick] (2.25, -1.125) -- (2.25, 3.375);
\draw[very thick] (-1.125, 3.0) -- (3.375, 3.0);
\draw[very thick] (3.0, -1.125) -- (3.0, 3.375);
\draw[very thin] (-1.125, -0.375) -- (3.375, -0.375);
\draw[very thin] (-1.125, 2.625) -- (3.375, 2.625);
\draw[very thin] (2.625, -1.125) -- (2.625, 3.375);
\draw[very thin] (-0.375, -1.125) -- (-0.375, 3.375);
\draw[black!33, very thick] (-1.125, 1.125) -- (1.125, 1.125) -- (1.125, -1.125);
\draw[black!33, very thick] (1.125, 3.375) -- (1.125, 1.875) -- (1.875, 1.875) -- (1.875, 1.125) -- (3.375, 1.125);
\draw[gray, fill=white] (-0.75, -0.75) circle (0.3);
\node at (-0.75, -0.75) {$2$};
\draw[gray, fill=white] (-0.75, 0.0) circle (0.3);
\node at (-0.75, 0.0) {$2$};
\draw[gray, fill=white] (-0.75, 0.75) circle (0.3);
\node at (-0.75, 0.75) {$2$};
\draw[gray, fill=white] (-0.75, 1.5) circle (0.3);
\node at (-0.75, 1.5) {$0$};
\draw[gray, fill=white] (-0.75, 2.25) circle (0.3);
\node at (-0.75, 2.25) {$0$};
\draw[gray, fill=white] (-0.75, 3.0) circle (0.3);
\node at (-0.75, 3.0) {$0$};
\draw[gray, fill=white] (0.0, -0.75) circle (0.3);
\node at (0.0, -0.75) {$2$};
\draw[black, fill=white] (0.0, 0.0) circle (0.3);
\node at (0.0, 0.0) {$2$};
\draw[black, fill=white] (0.0, 0.75) circle (0.3);
\node at (0.0, 0.75) {$2$};
\draw[black, fill=white] (0.0, 1.5) circle (0.3);
\node at (0.0, 1.5) {$0$};
\draw[black, fill=white] (0.0, 2.25) circle (0.3);
\node at (0.0, 2.25) {$0$};
\draw[gray, fill=white] (0.0, 3.0) circle (0.3);
\node at (0.0, 3.0) {$0$};
\draw[gray, fill=white] (0.75, -0.75) circle (0.3);
\node at (0.75, -0.75) {$2$};
\draw[black, fill=white] (0.75, 0.0) circle (0.3);
\node at (0.75, 0.0) {$2$};
\draw[black, fill=white] (0.75, 0.75) circle (0.3);
\node at (0.75, 0.75) {$2$};
\draw[black, fill=white] (0.75, 1.5) circle (0.3);
\node at (0.75, 1.5) {$0$};
\draw[black, fill=white] (0.75, 2.25) circle (0.3);
\node at (0.75, 2.25) {$0$};
\draw[gray, fill=white] (0.75, 3.0) circle (0.3);
\node at (0.75, 3.0) {$0$};
\draw[gray, fill=white] (1.5, -0.75) circle (0.3);
\node at (1.5, -0.75) {$0$};
\draw[black, fill=white] (1.5, 0.0) circle (0.3);
\node at (1.5, 0.0) {$0$};
\draw[black, fill=white] (1.5, 0.75) circle (0.3);
\node at (1.5, 0.75) {$0$};
\draw[black, fill=white] (1.5, 1.5) circle (0.3);
\node at (1.5, 1.5) {$0$};
\draw[black, fill=white] (1.5, 2.25) circle (0.3);
\node at (1.5, 2.25) {$-2$};
\draw[gray, fill=white] (1.5, 3.0) circle (0.3);
\node at (1.5, 3.0) {$-2$};
\draw[gray, fill=white] (2.25, -0.75) circle (0.3);
\node at (2.25, -0.75) {$0$};
\draw[black, fill=white] (2.25, 0.0) circle (0.3);
\node at (2.25, 0.0) {$0$};
\draw[black, fill=white] (2.25, 0.75) circle (0.3);
\node at (2.25, 0.75) {$0$};
\draw[black, fill=white] (2.25, 1.5) circle (0.3);
\node at (2.25, 1.5) {$-2$};
\draw[black, fill=white] (2.25, 2.25) circle (0.3);
\node at (2.25, 2.25) {$-2$};
\draw[gray, fill=white] (2.25, 3.0) circle (0.3);
\node at (2.25, 3.0) {$-2$};
\draw[gray, fill=white] (3.0, -0.75) circle (0.3);
\node at (3.0, -0.75) {$0$};
\draw[gray, fill=white] (3.0, 0.0) circle (0.3);
\node at (3.0, 0.0) {$0$};
\draw[gray, fill=white] (3.0, 0.75) circle (0.3);
\node at (3.0, 0.75) {$0$};
\draw[gray, fill=white] (3.0, 1.5) circle (0.3);
\node at (3.0, 1.5) {$-2$};
\draw[gray, fill=white] (3.0, 2.25) circle (0.3);
\node at (3.0, 2.25) {$-2$};
\draw[gray, fill=white] (3.0, 3.0) circle (0.3);
\node at (3.0, 3.0) {$-2$};
\end{scope}
        \end{tikzpicture}
        \subcaption{$\tilde f$}
        \label{fig_small_example_B}
      \end{subfigure}\\[1em]
      \begin{subfigure}{.35\textwidth}
        \centering
        \begin{tikzpicture}
          \begin{scope}
\clip (-0.375, -1.125) -- (2.625, -1.125) -- (2.625, 3.375) -- (-0.375, 3.375) -- cycle;
\draw[very thick] (-0.375, 0.0) -- (2.625, 0.0);
\draw[very thick] (0.0, -0.375) -- (0.0, 2.625);
\draw[very thick] (-0.375, 0.75) -- (2.625, 0.75);
\draw[very thick] (0.75, -0.375) -- (0.75, 2.625);
\draw[very thick] (-0.375, 1.5) -- (2.625, 1.5);
\draw[very thick] (1.5, -0.375) -- (1.5, 2.625);
\draw[very thick] (-0.375, 2.25) -- (2.625, 2.25);
\draw[very thick] (2.25, -0.375) -- (2.25, 2.625);
\draw[black, fill=white] (0.0, 0.0) circle (0.3);
\node at (0.0, 0.0) {$3$};
\draw[black, fill=white] (0.0, 0.75) circle (0.3);
\node at (0.0, 0.75) {$1.5$};
\draw[black, fill=white] (0.0, 1.5) circle (0.3);
\node at (0.0, 1.5) {$2$};
\draw[black, fill=white] (0.0, 2.25) circle (0.3);
\node at (0.0, 2.25) {$2.5$};
\draw[black, fill=white] (0.75, 0.0) circle (0.3);
\node at (0.75, 0.0) {$1.5$};
\draw[black, fill=white] (0.75, 0.75) circle (0.3);
\node at (0.75, 0.75) {$2$};
\draw[black, fill=white] (0.75, 1.5) circle (0.3);
\node at (0.75, 1.5) {$0.5$};
\draw[black, fill=white] (0.75, 2.25) circle (0.3);
\node at (0.75, 2.25) {$1$};
\draw[black, fill=white] (1.5, 0.0) circle (0.3);
\node at (1.5, 0.0) {$2$};
\draw[black, fill=white] (1.5, 0.75) circle (0.3);
\node at (1.5, 0.75) {$2.5$};
\draw[black, fill=white] (1.5, 1.5) circle (0.3);
\node at (1.5, 1.5) {$1$};
\draw[black, fill=white] (1.5, 2.25) circle (0.3);
\node at (1.5, 2.25) {$1.5$};
\draw[black, fill=white] (2.25, 0.0) circle (0.3);
\node at (2.25, 0.0) {$2.5$};
\draw[black, fill=white] (2.25, 0.75) circle (0.3);
\node at (2.25, 0.75) {$1$};
\draw[black, fill=white] (2.25, 1.5) circle (0.3);
\node at (2.25, 1.5) {$1.5$};
\draw[black, fill=white] (2.25, 2.25) circle (0.3);
\node at (2.25, 2.25) {$2$};
\end{scope}
        \end{tikzpicture}
        \subcaption{$g$}
        \label{fig_small_example_gA}
      \end{subfigure}
      \qquad
      \begin{subfigure}{.35\textwidth}
        \centering
        \begin{tikzpicture}
          \begin{scope}
\clip (-1.125, -1.125) -- (3.375, -1.125) -- (3.375, 3.375) -- (-1.125, 3.375) -- cycle;
\draw[very thick] (-1.125, -0.75) -- (3.375, -0.75);
\draw[very thick] (-0.75, -1.125) -- (-0.75, 3.375);
\draw[very thick] (-1.125, 0.0) -- (3.375, 0.0);
\draw[very thick] (0.0, -1.125) -- (0.0, 3.375);
\draw[very thick] (-1.125, 0.75) -- (3.375, 0.75);
\draw[very thick] (0.75, -1.125) -- (0.75, 3.375);
\draw[very thick] (-1.125, 1.5) -- (3.375, 1.5);
\draw[very thick] (1.5, -1.125) -- (1.5, 3.375);
\draw[very thick] (-1.125, 2.25) -- (3.375, 2.25);
\draw[very thick] (2.25, -1.125) -- (2.25, 3.375);
\draw[very thick] (-1.125, 3.0) -- (3.375, 3.0);
\draw[very thick] (3.0, -1.125) -- (3.0, 3.375);
\draw[very thin] (-1.125, -0.375) -- (3.375, -0.375);
\draw[very thin] (-1.125, 2.625) -- (3.375, 2.625);
\draw[very thin] (2.625, -1.125) -- (2.625, 3.375);
\draw[very thin] (-0.375, -1.125) -- (-0.375, 3.375);
\draw[black!66, very thick] (-1.125, 0.375) -- (0.375, 0.375) -- (0.375, -1.125);
\draw[black!66, very thick] (0.375, 3.375) -- (0.375, 1.125) -- (1.875, 1.125) -- (1.875, 0.375) -- (3.375, 0.375);
\draw[gray, fill=white] (-0.75, -0.75) circle (0.3);
\node at (-0.75, -0.75) {$3$};
\draw[gray, fill=white] (-0.75, 0.0) circle (0.3);
\node at (-0.75, 0.0) {$3$};
\draw[gray, fill=white] (-0.75, 0.75) circle (0.3);
\node at (-0.75, 0.75) {$1$};
\draw[gray, fill=white] (-0.75, 1.5) circle (0.3);
\node at (-0.75, 1.5) {$1$};
\draw[gray, fill=white] (-0.75, 2.25) circle (0.3);
\node at (-0.75, 2.25) {$1$};
\draw[gray, fill=white] (-0.75, 3.0) circle (0.3);
\node at (-0.75, 3.0) {$1$};
\draw[gray, fill=white] (0.0, -0.75) circle (0.3);
\node at (0.0, -0.75) {$3$};
\draw[black, fill=white] (0.0, 0.0) circle (0.3);
\node at (0.0, 0.0) {$3$};
\draw[black, fill=white] (0.0, 0.75) circle (0.3);
\node at (0.0, 0.75) {$1$};
\draw[black, fill=white] (0.0, 1.5) circle (0.3);
\node at (0.0, 1.5) {$1$};
\draw[black, fill=white] (0.0, 2.25) circle (0.3);
\node at (0.0, 2.25) {$1$};
\draw[gray, fill=white] (0.0, 3.0) circle (0.3);
\node at (0.0, 3.0) {$1$};
\draw[gray, fill=white] (0.75, -0.75) circle (0.3);
\node at (0.75, -0.75) {$1$};
\draw[black, fill=white] (0.75, 0.0) circle (0.3);
\node at (0.75, 0.0) {$1$};
\draw[black, fill=white] (0.75, 0.75) circle (0.3);
\node at (0.75, 0.75) {$1$};
\draw[black, fill=white] (0.75, 1.5) circle (0.3);
\node at (0.75, 1.5) {$-1$};
\draw[black, fill=white] (0.75, 2.25) circle (0.3);
\node at (0.75, 2.25) {$-1$};
\draw[gray, fill=white] (0.75, 3.0) circle (0.3);
\node at (0.75, 3.0) {$-1$};
\draw[gray, fill=white] (1.5, -0.75) circle (0.3);
\node at (1.5, -0.75) {$1$};
\draw[black, fill=white] (1.5, 0.0) circle (0.3);
\node at (1.5, 0.0) {$1$};
\draw[black, fill=white] (1.5, 0.75) circle (0.3);
\node at (1.5, 0.75) {$1$};
\draw[black, fill=white] (1.5, 1.5) circle (0.3);
\node at (1.5, 1.5) {$-1$};
\draw[black, fill=white] (1.5, 2.25) circle (0.3);
\node at (1.5, 2.25) {$-1$};
\draw[gray, fill=white] (1.5, 3.0) circle (0.3);
\node at (1.5, 3.0) {$-1$};
\draw[gray, fill=white] (2.25, -0.75) circle (0.3);
\node at (2.25, -0.75) {$1$};
\draw[black, fill=white] (2.25, 0.0) circle (0.3);
\node at (2.25, 0.0) {$1$};
\draw[black, fill=white] (2.25, 0.75) circle (0.3);
\node at (2.25, 0.75) {$-1$};
\draw[black, fill=white] (2.25, 1.5) circle (0.3);
\node at (2.25, 1.5) {$-1$};
\draw[black, fill=white] (2.25, 2.25) circle (0.3);
\node at (2.25, 2.25) {$-1$};
\draw[gray, fill=white] (2.25, 3.0) circle (0.3);
\node at (2.25, 3.0) {$-1$};
\draw[gray, fill=white] (3.0, -0.75) circle (0.3);
\node at (3.0, -0.75) {$1$};
\draw[gray, fill=white] (3.0, 0.0) circle (0.3);
\node at (3.0, 0.0) {$1$};
\draw[gray, fill=white] (3.0, 0.75) circle (0.3);
\node at (3.0, 0.75) {$-1$};
\draw[gray, fill=white] (3.0, 1.5) circle (0.3);
\node at (3.0, 1.5) {$-1$};
\draw[gray, fill=white] (3.0, 2.25) circle (0.3);
\node at (3.0, 2.25) {$-1$};
\draw[gray, fill=white] (3.0, 3.0) circle (0.3);
\node at (3.0, 3.0) {$-1$};
\end{scope}
        \end{tikzpicture}
        \subcaption{$\tilde g$}
        \label{fig_small_example_gB}
      \end{subfigure}%
      \caption
    {Two neighbouring $\p\n$-functions with $\p=(3,3)$ and $\n=(1,1)$.
    Subfigures~\ref{fig_small_example_A} and~\ref{fig_small_example_gA}  present height functions, which are defined 
    on the discrete torus.
    Subfigures~\ref{fig_small_example_B} and~\ref{fig_small_example_gB} present their lifts,
    which are defined on $\mathbb Z^d$.
    The boundaries of the associated 
    stepped surfaces---defined later---are drawn as grey lines.
    Moving four vertices left or down equates to increasing the height by two,
    owing to the periodicity of the original torus.
    }
    \label{fig_small_example}
    \end{center}
  \end{figure}

\subsection{Periodic height functions}
\label{whatever}
Fix a dimension $d\in\N$ throughout this paper.
The random walk of interest depends on the choice of two parameters $\p,\n\in\mathbb N^d$.
The parameter $\n$ is always fixed,
and we consider $\p$ fixed, at least until further notice.

We start with a description of the discrete torus.
Define $\s:=\p+\n$; the \emph{$\s$-torus} is the natural square lattice 
graph on the vertex set $\V:=\prod_{i=1}^d(\Z/\s_i\Z)$.
Write $\0$ for the natural zero element in $\V$.
If we write $(\e_i)_{1\leq i\leq d}$ for the standard basis of $\R^d$,
then the edges of the torus are given by $\E(\V):=\{\{\x,\x+\e_i\}:\x\in \V,1\leq i\leq d\}$.
Another way of introducing the $\s$-torus is by writing $\sigma=\sigma_\s$ for
the unique linear isomorphism of $\R^d$
that maps $\e_i$ to $\s_i\e_i$,
and then obtaining the graph by taking the quotient of 
the square lattice graph $\Z^d$
by $\sigma\Z^d$.
The map $\sigma$ is called the \emph{scaling map}.

We are now ready to introduce \emph{$\p\n$-periodic height functions},
which are height functions defined on this torus.
They are identical in spirit to the height functions described in the introduction,
with each edge oriented in one of two possible directions: \emph{up} or \emph{down}.
The parameters $\p$ and $\n$ encode how many edges of each orientation
are forced to appear in each height function. 
Formally, a \emph{$\p\n$-periodic height function}
or simply a \emph{$\p\n$-function} is a map $f:\V\to\R$ that satisfies,
for each $\{\x,\x+\e_i\}\in \E(\V)$,
\[\textstyle f(\x+\e_i)-f(\x)=-\frac{\p_i-\n_i}{\s_i}\pm 1.\]
We shall see shortly that the fraction,
which did not appear in the original definition in the introduction,
conveniently encodes the parameters $\p$ and $\n$
into the periodic boundary conditions imposed on the model.

A $\p\n$-function should be thought of as assigning a height to
each vertex of the $\s$-torus; see Subfigure~\ref{fig_small_example_A}
for a small numerical example and
Figure~\ref{fig_large_example} for a height function on a larger torus.
Write, with slight abuse of notation, $\nabla f$ for the unique
map $\nabla f:\E(\V)\to\{-1,1\}$
such that, for each $e=\{\x,\x+\e_i\}\in \E(\V)$,
\[\textstyle f(\x+\e_i)-f(\x)=-\frac{\p_i-\n_i}{\s_i}+\nabla f(e).\]
If $\nabla f(e)=1$ then the edge $e$ is called an
\emph{up step} for $f$, otherwise it is called a \emph{down step}.
Fix for now a reference point $\x\in \V$ and a direction $1\leq i\leq d$.
There are $\s_i$ edges of the form $\{\x+k\e_i,\x+(k+1)\e_i\}$;
they belong to a circular walk around the torus.
Summing the previous equation over these $\s_i$ edges gives
\begin{multline*}
   \sum\nolimits_{k=0}^{\s_i-1} f(\x+(k+1)\e_i)-f(\x+k\e_i)
  \\ =-(\p_i-\n_i)+\sum\nolimits_{k=0}^{\s_i-1}\nabla f(\{\x+k\e_i,\x+(k+1)\e_i\}).
\end{multline*}
Both sides of this equation equal zero as the terms in the sum on the left cancel.
Therefore this collection of $\s_i$ edges
must consist of exactly $\p_i$ up steps and exactly $\n_i$ down steps.
The interest in this paper is in small values for the entries of $\n$
and large values for the entries of $\p$.
We shall see that the down steps of a $\p\n$-function are organised into \emph{fractures},
and $\n_i$ is the number of times that a circular walk around the $\s$-torus
in direction $i$ intersects such a fracture, see Figures~\ref{fig_small_example} and~\ref{fig_large_example}.

% !TEX root = ../ms.tex
\begin{figure}
\begin{center}
\includegraphics[width=0.8\textwidth]{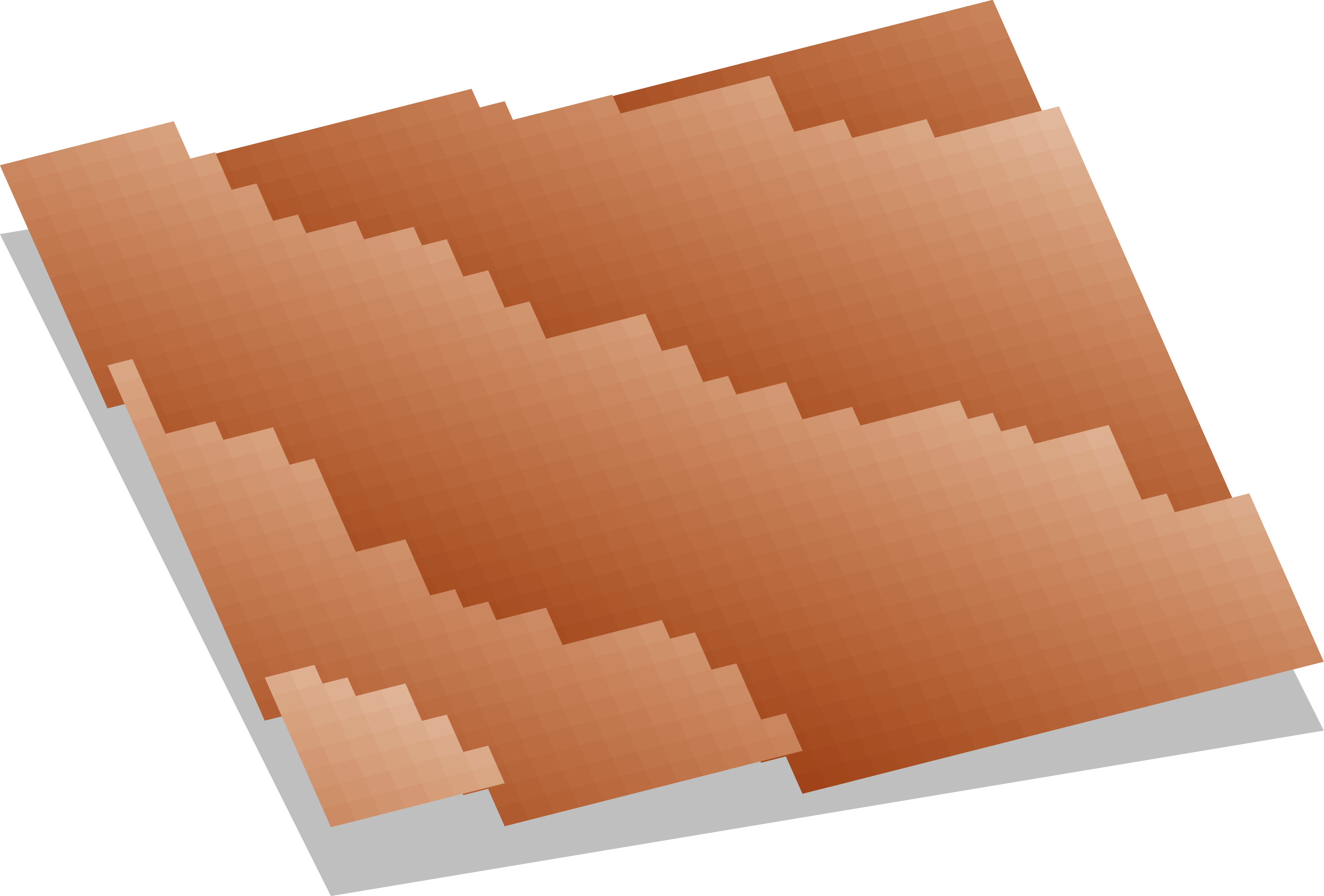}
\caption
[A $\p\n$-function]
{A $\p\n$-function with $\p=(38,38)$ and $\n=(2,2)$}
\label{fig_large_example}
\end{center}
\end{figure}

\subsection{A walk on height functions}
The previous provides the static picture of a single $\p\n$-function;
we now introduce a locally finite graph on the set of $\p\n$-functions
so that we can construct the random walk on $\p\n$-functions.
Say that two $\p\n$-functions $f_1$ and $f_2$ are \emph{neighbours},
and write $f_1\sim f_2$, whenever
 $|f_1-f_2|$ is identically equal to $1$.
 % Write $\mathcal N(f_1)$ for the set of neighbours of
 % $f_1$, and remark that $\mathcal N(f_1)$ is always nonempty and finite.
 % \color{blue} Do we need $\mathcal N(f_1)$? \color{black}
 Note that each $\p\n$-function has finitely many neighbours.
 This means that we can consistently define a random walk
 $X=(X_k)_{k\geq 0}$ on the set of $\p\n$-functions;
 the next state of the walk is chosen uniformly at random from
 the neighbours of the current state.
 We impose that $X_0(\0)=0$,
 so that $X_k(\0)\in\Z$ for all $k\geq 0$ almost surely.
 For simplicity, we shall restrict ourselves
 to studying $\p\n$-functions $f$
 with $f(\0)\in\Z$ in the sequel.

 \subsection{Main result}
 Say that two $\p\n$-functions \emph{have the same shape} if they differ
  by a constant. \emph{Shapes} are thus equivalence classes of $\p\n$-functions.
  There are finitely many shapes as the shape
  of $f$ is determined entirely by $\nabla f$.
 We show furthermore that the random walk $X$ is irreducible, regardless of the choice of  $\p$ and $\n$.
 Arguments similar to those presented
 in \cite{BCEN,EGN} imply that the law of
 \begin{equation*}\textstyle
   \left(\frac1{\sqrt n}X_{\lfloor nt\rfloor}(\0)\right)_{t\in [0,1]}
 \end{equation*}
 converges to that of a Brownian motion of some diffusivity $\alpha(\p,\n)$ as $n\to\infty$,
 and---as the notation suggests---this diffusivity must depend on $\p$ and $\n$ only.
 Espinasse, Guillotin-Plantard, and Nadeau \cite{EGN} provide an explicit formula for $\alpha(\p,\n)$  whenever $d=1$.
 The following theorem, which is the main result of this paper, holds in any dimension $d$.

 \begin{theorem}\label{main}
 	In any dimension $d\in\mathbb N$, we have
 \[\lim_{\p\to\infty}\alpha(\p,\n)=\frac{1}{1+2\operatorname{gcd}\n}.\]
 Here we write $\p\to\infty$ for $\p_1,...,\p_d\to\infty$, and $\operatorname{gcd}\n$ is short for $\operatorname{gcd}(\n_1,...,\n_d)$.
 \end{theorem}

 The theorem is interesting for the following reasons.
 If $d=1$ then
 the diffusivity $\alpha(\p,\n)$ of the walk  depends on two parameters
 $\p,\n\in \mathbb{N}$,
 and the value of $\alpha(\p,\n)$ is decreasing in both parameters; see \cite{EGN} for
 the explicit values of $\alpha(\p,\n)$.
 One may expect the same behaviour whenever $d>1$,
 because an increase in the entries of $\n$ associates
 with an increase in the number of distinct
 shapes, which in turn one may expect to slow down the random walk and decrease the diffusivity.
 This intuition is wrong:
 if the entries of $\p$ are large
 and the entries of $\n$ small,
 then the theorem tells us that the diffusivity $\alpha(\p,\n)$ depends mostly on $\gcd\n$.
 Increasing $\n_1$ will therefore either increase or decrease $\alpha(\p,\n)$,
 depending on how the increase in $\n_1$ changes the value of $\gcd\n$.
 In the sequel we always consider $d$ and $\n$ fixed,
 and $\p$ is considered fixed unless we explicitly take the limit $\p\to\infty$.

\subsection{Proof idea}
 Each shape is characterised by the locations of its down steps,
 which in turn are organised into fractures.
 In dimension $d=2$ these fractures are the paths on the torus which suggestively appear in Figure~\ref{fig_large_example}
 (as $\gcd\n=2$, there are two such paths in this example).
 We first analyse the uniform distribution on the set of shapes.
 As $\n$ remains fixed and $\p$ grows large,
 each path behaves like a random walk (conditioned to return to a specific value to account for boundary conditions).
 Such random walks become straight lines in the mesh limit,
 because the typical deviations of the random walk are of a lower order than the rescaling factor.
 The geometric behaviour of each random walk then depends mostly on its starting point,
 which may also be chosen uniformly at random.
 In this two-dimensional setup,
 it is straightforward to work out that two indepently chosen paths 
 do not intersect one another with high probability
 (because they are flat with high probability,
 which means that they can only intersect if the starting points are chosen 
 in a highly unfortunate manner.)
 Thus, the uniform distribution on the set of shapes 
 is very close to a distribution in which the different fractures 
 are identically and independently distributed.
 All of the above will be presented in arbitrary dimension,
 where random walks turning into straight lines become 
 random surfaces turning into flat hypersurfaces of codimension one.
 In the mesh limit,
 the starting points of the hypersurfaces are described by uniform random variables
 in the unit circle $\mathbb R/\mathbb Z$.
 This turns the discrete $d$-dimensional model 
 into a continuous one-dimensional model.
 Finally, we study the uniform distribution on the set of adjacent pairs 
 of shapes.
 This allows for an analysis similar to the analysis above,
 and in particular also yields a continuous one-dimensional description of the model in the mesh limit.
 This model is straightforwardly solved analytically to produce 
 the formula on the right in Theorem~\ref{main}.
 A substantial amount of effort is spent on demonstrating 
 that the diffusivity of the discrete process is preserved in the mesh limit,
 that is, that the diffusivity provided by the continuous model 
 is indeed equal to the limit on the left in the display in Theorem~\ref{main}.

\subsection{Overview}
Section~\ref{sec:combinatorial} studies the combinatorial structure of the random walk in terms of fractures.
Section~\ref{sec:embedding} describes how fractures are embedded in strips.
Strips turn out to be an effective tool in controlling the macroscopic behaviour
of fractures.
Section~\ref{sec:random} analyses the behaviour of a uniformly random fracture.
It uses ideas from~\cite{SELF} to demonstrate that fractures are asymptotically
\emph{flat}, in the sense that the minimal strips containing each fracture---as defined in Section~\ref{sec:embedding}---are thin in the mesh limit.
In Section~\ref{sec:process_decomposition_and_martingale_property},
the original random walk on height functions is
split into two processes: a martingale, and a sort of remainder process.
In Section~\ref{sec:final_proof},
we prove that the diffusivity of the remainder process is negligible in the mesh limit,
and we reduce the martingale to a one-dimensional system
with the desired diffusivity.

\subsection{Notation}
\label{subsec:notation}
It will be useful to introduce some basic tools.
These are all fixed, that is, they are allowed to depend only on the choice of the parameters $\p$ and $\n$.
Recall that we have already introduced the vector $\s:=\p+\n$, the linear isomorphism $\sigma$ of $\R^d$ that maps $\e_i$ to $\s_i\e_i$,
and the $\s$-torus $(\V,\E(\V))$ which is obtained
by taking the quotient of the square lattice $\Z^d$ with $\sigma\Z^d$.
If $f$ is a $\p\n$-function, then write $\hat f$ for its average:
\[
  \hat f:=\frac1{|\V|}\sum_{\x\in \V}f(\x).
\]
The difference $(X_n(\mathbf 0)-\hat X_n)_{n\geq 0}$
is bounded uniformly (for fixed $\p$ and $\n$),
and therefore one may replace $X(\mathbf 0)$ by $\hat X$
in the definition of $\alpha(\p,\n)$.

% Write $\nu:=\sigma^{-1}$:
% the \emph{normalisation map}.
% It will be convenient to work in four distinct
% continuous spaces, see the diagram of Figure~\ref{fig:diagram}.
% The two spaces on the right are tori,
% scaled and unscaled,
% and the two spaces on the left
% are the (scaled and unscaled) universal covering spaces of these tori.
% Here $\pi^\nu$ and $\pi^\sigma$ denote the natural projection maps from $\R^d$
% to $\R^d/\Z^d$ and $\R^d/\sigma\Z^d$ respectively.
% The discrete $\s$-torus has a natural embedding in $\R^d/\sigma\Z^d$,
% and the
%  linear isomorphisms $\sigma$ and $\nu$ extend naturally
% to bijections between $\R^d/\Z^d$ and $\R^d/\sigma\Z^d$.
% The torus $\R^d/\Z^d$ is occasionally more convenient to work with because it is normalised and thus independent of the choice of $\p$ and $\n$.
% % The two remaining spaces are both $\R^d$,
% % but they appear separately in the commutative diagram of
% % Figure~\ref{fig:diagram}.
% Many objects in this article are naturally associated with one of the four
% spaces in the diagram, and the elementary maps between these spaces are useful for manipulating
% said objects.
% \input{figures/diagram.tex}

We shall write $(\a,\b)$ for the usual inner product of $\a,\b\in\R^d$,
and $\a^\perp$ for the hyperplane in $\R^d$ that is orthogonal to $\a$.
Of particular interest is $\n^\perp$,
and we shall write $N$ for $\n^\perp\cap\Z^d$:
the $\Z$-module of integral vectors orthogonal to $\n$.
The $\Z$-module $N$ spans $\n^\perp$ because the entries of $\n$ are integers.
We shall also write $\f$ for a fixed vector in $\Z^d$
such that $(\f,\n)=\gcd\n$.
Such a vector must exist by the definition of the greatest
common divisor.
Each
vector in $\Z^d$ can be written uniquely as the sum of
some vector $\a\in N$ and a vector of the form $b\f$ with $b\in\Z$.

On several occasions, it will be useful to move between the discrete $\s$-torus
and the \emph{unit torus},
that is, the set $\T:=\R^d/\Z^d$.
This is where the map $\sigma$ and its inverse play an important role.
We shall also sometimes work with the universal covers of these structures.
Define the map $\pi:\R^d\to\T$ as follows:
for $\x\in\R^d$, let
\[
    \pi(\x):=\sigma^{-1}\x+\Z^d\in\T.
\]
Note that $\pi^{-1}(\x)$ is automatically $\sigma\Z^d$-invariant
for any $\x\in\T$.

Write $\theta$ for the map
$\R^d\to\R^d,\x\mapsto\x+\sigma\f$.
We call $\theta$ a \emph{shift} because it simply translates each vector of $\R^d$ by $\sigma\f$; the latter is a vector with entries
$(\sigma\f)_i=\s_i\f_i$.
Finally, we shall conveniently abbreviate $\gcd\n$ to $|\n|$.

  % !TEX root = ../ms.tex

\section{The combinatorial structure of the random walk}
\label{sec:combinatorial}

This section formalises the notion of a fracture,
and we show that each shape can be represented by exactly $|\n|$ fractures.
This representation is then used to characterise the neighbourship relation $\sim$.
Finally, we demonstrate that the walk on $\p\n$-functions is
irreducible.

\subsection{Stepped surfaces and fractures}

Each of the two torus paths in Figure~\ref{fig_large_example} represents a fracture.
A fracture is essentially a hypersurface of codimension one which
separates two components which are connected by up steps.
The fracture itself consists of down steps.
There are two strategies by which fractures can be introduced formally.
First, one can define a graph structure on the set of down steps
(which serves as vertex set).
The fractures are then the connected components of this graph.
The connectivity of this graph is defined in terms of a natural local rule 
on the set of down steps.
Althought this strategy is perhaps more natural at first sight,
we prefer to work with another strategy
which very much simplifies the remainder of the analysis.
This second strategy is global,
in the sense that fractures are defined on the entire space directly without 
first defining a local rule and then seeing how this local rule permeates
throughout the space.
The global definition is natural after lifting each height function on the $\s$-torus to the universal 
cover $\mathbb Z^d$.

Before we proceed with the construction,
we mention that fractures are realised as \emph{stepped surfaces}.
In dimension two, a stepped surface can be thought of as a stack of unit squares 
such that its boundary is an infinite right-down path.
In dimension three, a stepped surface can be thought of as a stack of unit cubes 
whose boundary corresponds exactly to a lozenge tiling.
For an introduction into stepped surfaces as well as for some results that we shall appeal to here,
we refer to~\cite{SELF} and the references therein.

Each $\p\n$-function $f$ has the vertex set $\V$
of the $\s$-torus as its domain.
Stepped surfaces, however, are defined on the full square lattice $\mathbb Z^d$.
If $\tilde\x\in\Z^d$, then write $\x:=\tilde\x+\sigma\Z^d\in\V$
for the associated torus vertex.
% and is therefore naturally associated with the bottom-right
% space in Figure~\ref{fig:diagram}.
Let us first use the map $\tilde\x\mapsto\x$
to lift $f$
to a function $\tilde f$ on $\Z^d$.
% All other objects are then associated with the bottom-left space
% in Figure~\ref{fig:diagram}.

\begin{definition}
For $f$ a $\p\n$-function,
define the lift
$\tilde f$
by
\[
\tilde f:
\Z^d\to \Z,\,
\tilde\x\mapsto
 f(\x)
 -2(\tilde\x,\sigma^{-1}\n).
\]
\end{definition}

Note that the definition is as expected,
except that an extra correction term is introduced.
This correction term does not depend on $f$.
The purpose of the correction is to ensure that the discrete derivative 
of $\tilde f$ takes values in $\{-2,0\}$.
See Figure~\ref{fig_small_example} for an example of a height function and its lift.
By comparing the definition of $\tilde f$ with 
the definition of a $\p\n$-function in the introduction,
it is straightforward to derive the following proposition.

\begin{proposition}
  \label{propo_relation_f_tilde_f_derivative}
  If $f$
    is a  $\p\n$-function,
    then $\tilde f(\0)=f(\0)$, and for any $\tilde\x$
    and $i$,
    % $\x\in\mathbb{Z}^d$ and $1\leq i\leq d$,
\[\tilde f(\tilde\x+\e_i)-\tilde f(\tilde\x)=
-2\cdot 1(\nabla f(e)=-1)
=
-2\cdot 1(\text{$e$ is a down step for $f$}),
\]
where $e:=\{\x,\x+\e_i\}\in\torusedges$.
The function $\tilde f$ is the only function on $\mathbb{Z}^d$
with these properties,
and it is
 non-increasing in every coordinate.
\end{proposition}

Let us focus on the upper level sets
of the lift $\tilde f$.
Abbreviate $\{\x\in\Z^d:\tilde f(\x)\geq\lambda\}$
to $\{\tilde f\geq \lambda\}$ for any $\lambda\in\R$.
We shall soon see that such sets are \emph{stepped surfaces}
(which are yet to be defined).
We make the following simple preliminary observation, which is due entirely to
the fact that the discrete derivative of $\tilde f$
takes even values.

\begin{proposition}
  \label{propo_consider_only_even_for_lambda}
  We have $\{\{\tilde f\geq \lambda\}:\lambda\in\R\}=\{\{\tilde f\geq \lambda\}:\lambda\in a+2\Z\}$
  for any $\p\n$-function $f$ and for any $a\in\mathbb R$.
\end{proposition}

\begin{definition}
  A \emph{stepped surface} is a nonempty strict subset $A$
  of $\Z^d$
  such that $\mathbf x\in A$
  implies $\mathbf x-\mathbf e_i\in A$
  for every $1\leq i\leq d$.
  Informally, this means that each vertex in $A$ is well-supported by the vertices below it,
  which are also in $A$.
  If $M$ is a $\Z$-submodule of $\Z^d$,
  then
  the stepped surface $A$ is called \emph{$M$-invariant}
  if $A=A+\mathbf x$ for any $\mathbf x\in M$.
\end{definition}

If $\tilde f:\Z^d\to\R$ is an arbitrary function
that is non-increasing in every coordinate,
and $\lambda\in\R$,
then $\{\tilde f\geq\lambda\}$ is a stepped surface,
unless $\{\tilde f\geq\lambda\}\in\{\varnothing,\Z^d\}$.

\begin{lemma}
  \label{lemma_nature_of_pi_level_sets}
  If $f$ is a $\p\n$-function and $\lambda\in\R$,
  then $\{\tilde f\geq \lambda\}$ is a stepped surface.
  Moreover, if $\x\in\Z^d$,
  then
  $\{\tilde f\geq \lambda-2(\x,\n)\}=
  \{\tilde f\geq\lambda\}+\sigma\x
  $.
  In particular, $\{\tilde f\geq \lambda\}$ is a $\sigma N$-invariant stepped surface.
\end{lemma}

The lemma is quite natural:
the function $\tilde f$ is a lift of 
a function defined on the torus,
and this result expresses how the periodicity of the torus 
translates to symmetries of the upper level sets of $\tilde f$.  
In Figure~\ref{fig_small_example}, the set of vertices to the bottom-left of each grey line
is an upper level set and a stepped surface.

\begin{proof}[Proof of Lemma~\ref{lemma_nature_of_pi_level_sets}]
  We first claim that
  \begin{equation*}
    \tilde f (\y+\sigma\x)-\tilde f(\y)=-2(\x,\n)
  \end{equation*}
  whenever $\x,\y\in\Z^d$.
  The proof of the claim follows by integrating the discrete derivative 
  of $\tilde f$ along a path from $\y$ and $\y+\sigma\x$,
  and using the definition of a $\p\n$-function and
  Proposition~\ref{propo_relation_f_tilde_f_derivative}
  to see that this integral yields the expression on the right.
  The claim implies that $\tilde f$ takes arbitrarily small and
  large values,
  and consequently $\{\tilde f\geq \lambda\}\not\in\{\varnothing,\Z^d\}$.
  Therefore it must be a stepped surface.
  It follows from
  the claim
  that $\{\tilde f\geq \lambda-2(\x,\n))\}=
  \{\tilde f\geq \lambda\}+\sigma\x
  $
  for any $\x\in\Z^d$.
  Since $N\subset\Z^d$ was defined such that $(\x,\n)=0$
  for any $\x\in N$,
  this implies the remainder of the lemma.
\end{proof}

The previous lemma tells us that if $\{\tilde f\geq\lambda\}$ is a level set for $\tilde f$,
then so is $\{\tilde f\geq\lambda\}+\sigma\x$ for any $\x\in\Z^d$
(althought potentially for a different height $\lambda$ depending on $\x$).
Recall from Subsection~\ref{subsec:notation} that the set $\Z^d$ can be decomposed as the direct sum 
of the $\Z$-modules $N$ and $\mathbb Z\f$,
where $\mathbf f\in\mathbb Z^d$ is chosen such that
$(\mathbf f,\mathbf n)=|\n|:=\gcd\mathbf n$.
Recall also that $\theta$ denotes the shift $\mathbf x\mapsto \mathbf x+\sigma\mathbf f$.
Therefore $
\{\theta^kA:k\in\Z\}
=
\{A+\sigma\x:\x\in\Z^d\}
$
whenever $A$ is a $\sigma N$-invariant
stepped surface.
This motivates the formal definition of a fracture.

\begin{definition}
  A \emph{$\p\n$-fracture} or simply a \emph{fracture} is a set of stepped surfaces
  of the form $\{\theta^kA:k\in\Z\}$,
  where $A$ is a $\sigma N$-invariant stepped surface.
  Write $F(\p,\n)$ for the set of $\p\n$-fractures.
\end{definition}

We finally mention 
some simple relations between the different fractures corresponding to the same height function.

\begin{definition}
For any $A\subset\Z^d$,
write $\partial^* A$ for the \emph{vertex boundary} of $A$, that is, the set of vertices which are adjacent to $A$
in the square lattice.
If $A,B\subset \mathbb Z^d$,
then call $A$ a \emph{good subset}
of $B$, and write $A\sqsubset B$,
if  $A$ is a proper subset of $B$
with $\partial^* A\subset B$.
The relation $\sqsubset$ is a strict partial order on the subsets of $\Z^d$.
\end{definition}

\begin{lemma}
  \label{manystatementsaboutlevelsets}
  Let $f$ be a $\p\n$-function.
  Then all stepped surfaces of the form
  $\{\tilde f\geq \lambda\}$ are given by the chain
  \[\dots\sqsubset \{\tilde f\geq 4\}\sqsubset \{\tilde f\geq 2\}\sqsubset \{\tilde f\geq 0\}\sqsubset \{\tilde f\geq -2\}\sqsubset \{\tilde f\geq -4\}\sqsubset\cdots,\]
  where $\{\tilde f\geq \lambda-2k|\n|\}=\theta^k\{\tilde f\geq \lambda\}$ for each $k\in\Z$.
  In other words,
  the set of stepped surfaces $\{\{\tilde f\geq\lambda\}:\lambda\in\R\}$ can be written uniquely
  as the union of $|\n|$ distinct fractures,
  and for an appropriate enumeration $(F_1,\dots, F_{|\n|})$
  of these fractures and for an appropriate choice of representative stepped surface $A_i\in F_i$ for each fracture,
  the previous chain is identical to
  \[\dots\sqsubset\theta^{-1}A_{|\n|}\sqsubset A_1 \sqsubset A_2\sqsubset \dots\sqsubset A_{|\n|}\sqsubset \theta A_1\sqsubset \theta A_2\sqsubset \cdots.\]
\end{lemma}

\begin{proof}
  The relation $\{\tilde f\geq \lambda\}\sqsubset\{\tilde f\geq \lambda-2\}$
  for $\p\n$-functions $f$
  follows from the fact that the discrete derivative of $\tilde f$
  is bounded by $2$, see Proposition~\ref{propo_relation_f_tilde_f_derivative}.
  This proves the relations in the first display of the lemma;
  Proposition~\ref{propo_consider_only_even_for_lambda}
  asserts that the chain contains all stepped surfaces
  in $\{\{\tilde f\geq \lambda\}:\lambda\in\mathbb R\}$.
   Lemma~\ref{lemma_nature_of_pi_level_sets} asserts
  that
  $\{\tilde f\geq\lambda-2k|\n|\}=\theta^k\{\tilde f\geq \lambda\}$.
  For the rest of the lemma, the following choice of fractures
  and representative stepped surfaces
  suffices:
  define $A_i:=\{\tilde f\geq -2i\}$ and $F_i:=\{\theta^kA_i:k\in\mathbb Z\}$
  for $1\leq i\leq|\n|$.
  It is clear that the choice of fractures is unique, up to
  indexation.
\end{proof}

\subsection{Shapes}

\begin{definition}
  If $f$ is a $\p\n$-function,
  then
  the equivalence class
  $[f]$ defined by $[f]:=\{f+k:k\in\mathbb Z\}$
  is called the \emph{shape} of $f$.
  % Say that two $\p\n$-functions $f$ and $g$ \emph{have the same shape}
  % if they differ by a constant.
  % This defines an equivalence class;
  % write $[f]$ for the equivalence class of $f$,
  % that is, $[f]:=\{\dots,f-1,f,f+1,\dots\}$.
  % The equivalence class $[f]$ is called the \emph{shape}
  % of $f$, or simply a \emph{shape}.
  Write $S(\p,\n)$ for the set of shapes of $\p\n$-functions.
  Note that $S(\p,\n)$ is finite.
\end{definition}

The main purpose of this subsection is to prove the following lemma.
The lemma is the inverse statement of Lemma~\ref{manystatementsaboutlevelsets}:
it asserts that any consistent family of fractures corresponds to a unique shape.
Together, the two lemmas provide a bijection between shapes and families of fractures,
which is made explicit in Proposition~\ref{propo:bijection}.

\begin{lemma}
  \label{iffinethenheightfunction}
  Suppose given a collection of
  $|\n|$ distinct $\p\n$-fractures,
  such that for an appropriate enumeration $(F_i)_{1\leq i\leq|\n|}$ of these $|\n|$ fractures
  and for an appropriate choice of representatives $A_i\in F_i$,
  we have
  \[\dots\sqsubset\theta^{-1}A_{|\n|}\sqsubset A_1 \sqsubset A_2\sqsubset \dots\sqsubset A_{|\n|}\sqsubset \theta A_1\sqsubset \theta A_2\sqsubset \cdots.\]
  Then there is a $\p\n$-function $f$
  such that $\{\{\tilde f\geq \lambda\}:\lambda\in\R\}$
  equals $\cup_iF_i$.
  The choice of $f$ is unique up to an additive constant.
  In other words, the fractures $(F_i)_{i}$
  define a unique shape.
\end{lemma}

\begin{proof}
Define $A_{i+j|\n|}:=\theta^jA_i$ for each $1\leq i\leq|\n|$
and $j\in\Z$,
so that $(A_k)_{k\in\Z}=\cup_kF_k$
and $A_k\sqsubset A_{k+1}$ for each $k\in \Z$.
Clearly $\cap_kA_k=\varnothing$ and
$\cup_kA_k=\Z^d$,
so that the collection of sets $(B_k)_{k\in\Z}$ defined by $B_k:=A_k\smallsetminus A_{k-1}$ partitions $\Z^d$.
Define $\tilde g:\mathbb{Z}^d\to\mathbb{R}$ by
\[
  \tilde g
  :=
  \sum_{k\in\mathbb{Z}}-2k\cdot 1_{B_k}.
\]
It follows from the definition of $\tilde g$ that
$\{\tilde g\geq -2k \}= A_k$ for every $k\in\mathbb{Z}$.
It suffices to demonstrate that there exists a $\p\n$-function $f$ such that $\tilde f=\tilde g$.
For this it is enough to show that
\begin{enumerate}
  \item For any $\x\in\mathbb{Z}^d$ and $1\leq i\leq d$, we have $\tilde g(\x+\e_i)-\tilde g(\x)\in\{0,-2\}$,
  \item For any $\x,\y\in\mathbb{Z}^d$, we have $\tilde g(\x+\sigma\y)-\tilde g(\x)=-2(\y,\n)$.
\end{enumerate}

For the first statement, observe that
\[
  \tilde g(\x+\e_i)-\tilde g(\x)=-2|\{k\in\mathbb{Z}:\x\in A_k,\x+\e_i\not \in A_k\}|.
\]
The set in the display clearly contains at most
one element, since $\sqsubset$ is a strict
total order on $(A_k)_{k\in\Z}$.
% Suppose that the set in the display contains two
% distinct elements, say $k,k'\in\Z$ with $k<k'$,
% in order to derive a contradiction.
% Then $\x\in A_k$
% and $A_k\sqsubset A_{k'}$,
% and therefore $\x+\e_i\in A_{k'}$.
% This contradicts that $\x+\e_i\not\in A_{k'}$,
% which proves the first statement.
Now focus on the second statement.
The vector  $\y\in\mathbb Z^d$ can be written
uniquely as the sum of some vector $\a\in N$
and of the vector $b\f$, where $b:=(\y,\n)/|\n|\in\mathbb Z$.
We see that
\[
  A_k+\sigma\y=A_k+\sigma\a+\sigma b\f
  =A_k+\sigma b\f
  =\theta^bA_k
  =A_{k+b|\n|}.
\]
The second equality is due to $\sigma N$-invariance of $A_k$;
the final equality follows from the definition of $(A_k)_{k\in\Z}$.
The equality also implies that
$B_k+\sigma\y=B_{k+b|\n|}$.
If $\x\in B_k$,
then $\x+\sigma\y\in B_{k+b|\n|}$,
and therefore
\[\tilde g(\x+\sigma\y)-\tilde g(\x)=-2(k+b|\n|)+2k=-2b|\n|=-2(\y,\n).\]
This finishes the proof of the second statement.
It is clear that these fractures define a unique---up yo constant---function
$\tilde g$, and therefore they determine a unique shape.
\end{proof}

Before proceeding, let us write down some equivalent definitions of a shape,
which follow straightforwardly from the definitions and from the previous lemma.

\begin{proposition}
  % \label{boringequivalence}
  Let $f$ and $g$ be $\p\n$-functions.
  Then the following are equivalent:
  \begin{enumerate}
    \item The functions $f$ and $g$ have the same shape,
    \item The functions $\nabla f$ and $\nabla g$ are equal,
    \item The functions $\tilde f$ and $\tilde g$ differ by a constant,
    \item The set $\{\{\tilde f\geq\lambda\}:\lambda\in \mathbb{R}\}$ equals the set $\{\{\tilde g\geq\lambda\}:\lambda\in \mathbb{R}\}$.
  \end{enumerate}
\end{proposition}

The following proposition follows from Lemmas~\ref{manystatementsaboutlevelsets}
and~\ref{iffinethenheightfunction}.

\begin{proposition}
\label{propo:bijection}
  The set $S(\p,\n)$ is in bijection with
  the set
  \[
    \{
      F=(F_i)_i\subset F(\p,\n):
      \text{$|F|=|\n|$ and $\sqsubset$ restricts to
      a strict total order on $\cup_i F_i$}
    \}.
  \]
\end{proposition}

\begin{proof}
To apply the two lemmas, we need to verify one statement:
that the chain of Lemma~\ref{iffinethenheightfunction}
can be constructed from the collection of fractures $(F_i)_i$ whenever
$\sqsubset$ restricts to a strict total order on $\cup_iF_i$.
Let us induct on the number of fractures $n\leq|\n|$ for which we can
consistently construct the chain.
If $n=1$, then we only consider the first fracture $F_1$.
Write $A_1$ for some stepped surface in $F_1$.
Since we know (by assumption) that $\sqsubset$ restricts 
to a strict total order on $F_1=\{\theta^k A_1:k\in\mathbb Z\}$,
we must have $\theta^kA\sqsubset\theta^{k+1}A$ for each $k\in\mathbb Z$.
This proves the induction basis.
Now follows the induction step:
we consider the proposition proven for the first $n$ fractures,
and focus on the first $n+1$.
By induction, we can reorder the indices of the first $n$ fractures and select a representative 
stepped surface $A_i\in F_i$ for each fracture such that 
  \[
    \dots\sqsubset\theta^{-1}A_{n}\sqsubset A_1 \sqsubset A_2\sqsubset \dots\sqsubset A_{n}\sqsubset \theta A_1\sqsubset \theta A_2\sqsubset \cdots.
  \]
  The objective is to consistently fit the extra fracture $F_{n+1}$ into this chain.
  Select $A'\in F_{n+1}$.
  Note that $\theta^{-k}A_1\sqsubset A'\sqsubset \theta^kA_1$ for $k$ sufficiently large.
  Since $\sqsubset$ restricts to a total order on the union of the fractures,
  this means that $A'$ fits somewhere into the chain in the previous display.
  The other stepped surfaces in $F_{n+1}$ may be consistently 
  fit into the chain simply by observing that $B\sqsubset C$
  if and only if $\theta B\sqsubset\theta C$ for some arbitrary stepped surfaces 
  $B$ and $C$.
  It suffices to relabel the first $n+1$ fractures and representative stepped surfaces appropriately
  to obtain the chain in Lemma~\ref{iffinethenheightfunction} with the correct 
  labels.
\end{proof}

  % !TEX root = ../ms.tex

\subsection{The neighbourship relation}

Recall that two $\p\n$-functions $f$ and $g$ are neighbours 
if $|f-g|$ is identically equal to $1$,
and that we write $f\sim g$ in this case.
We first translate this relation to shapes.

\begin{definition}
  For $V,W\in S(\p,\n)$,
   write $V\sim W$,
  and say that $V$ and $W$ are \emph{neighbours},
  if there exist $\p\n$-functions $f\in V$ and $g\in W$
  such that $f\sim g$.
\end{definition}

% The proposition tells us that if $A$ is a shape, then the set
% $\{\Pi_{f}(\lambda):\lambda\in \mathbb{R}\}$
% is independent of the choice of $f\in A$.
% Therefore the following definition is consistent;
% the set $\Lambda(A)$ does not depend on the choice of the height function
% representing $A$.
%
% \begin{definition}
%   Let $A$ be a shape and $f\in A$.
%   Then write $\Lambda(A)$ for
%   the set of
%   $\gcd\n$ elements $(\Lambda_i)_{1\leq i\leq \gcd\n}\subset L^{\p\n}$
%   such that
%    $\{\Pi_{f}(\lambda):\lambda\in \mathbb{R}\}=\cup_i\Lambda_i$.
% \end{definition}
%
% This definition well-defines $\Lambda(A)$ because of the preceding comment
% and Corollary~\ref{iwdunouafnawdawadw}.
% The content of the following proposition regarding $[f]\neq B$ follows easily
% by comparing the values of $f$ and $g$ at two points
% $\x,\x+e_i\in V_\t$ such that $\operatorname{Sgn}_f(\{\x,\x+e_i\})\neq\operatorname{Sgn}_g(\{\x,\x+e_i\})$.

We first state a simple observation which will be useful later.

\begin{proposition}
  \label{waarbenikmeebezig}
  Let $f$ be a height function and $W$ a shape.
  If $[f]=W$ then the neighbours of $f$ with shape $W$ are $f-1 $ and $f+1$.
  If $[f]\neq W$ then there exists at most one neighbour $g\sim f$
  with $g\in W$.
\end{proposition}

\begin{proof}
  The first statement is obvious.
  For the second statement,
  suppose that $W\sim [f]$, but $W\neq [f]$.
  If $g\in W$, then $f-g$ is not constant,
  and therefore there exists at most one $k\in\mathbb Z$
  such that $g+k-f$ takes values in $\{-1,1\}$.
  In that case, $g+k$ is the unique neighbour of $f$
  contained in $W$.
\end{proof}

The proposition says that the height difference $g-f$
is independent of the choice of the height functions to represent the shapes 
whenever the two shapes are neighbours of one another and distinct.
We further characterise this height difference in the following lemma.

% Suppose now that we are given
% two distinct neighbouring shapes $F\sim G$.
% Then there exist height functions $f\in F$ and $g\in G$ such that $f\sim g$.
% We claim that the difference $g-f$ does not depend on the choice of functions
% $f$ and $g$.
% If we had chosen $f+k$ from $A$ instead of $f$ (for some $k\in \mathbb{Z}$),
% then there is only one option for the choice of height function from $G$:
% it must be $g+k$. But adding $k$ to both $f$ and $g$ does not change
% the difference $g-f$. This proves the claim.

\begin{lemma}
  \label{awinbwanoafbw}
  Let $V$ and $W$ be shapes, say
  with $f\in V$
  and $g\in W$.
  Then $V\sim W$ if and only if
  \begin{equation*}
    % \label{intertwine}
    \cdots\subset A_{-1} \subset B_{-1} \subset A_{0} \subset B_{0} \subset A_{1} \subset B_{1} \subset \cdots
  \end{equation*}
  for some indexations $(A_k)_{k\in\mathbb{Z}}$
  and $(B_k)_{k\in\mathbb{Z}}$
  respectively
  of the sets $\{\{\tilde f\geq\lambda\}:\lambda\in\mathbb R\}$
  and $\{\{\tilde g\geq\lambda\}:\lambda\in\mathbb R\}$
  with $A_{k}\sqsubset A_{k+1}$ and
   $B_{k}\sqsubset B_{k+1}$ for each $k$.
  Now suppose that $V\sim W$ with $V\neq W$,
  and pick $f\in V$ and $g\in W$ such that $f\sim g$.
  Define
  \[
    \V(fg):=\cup_k\{\x:\text{$\tilde\x\in A_k\smallsetminus B_{k-1}$}\},\qquad
    \V(gf):=\cup_k\{\x:\text{$\tilde\x\in B_k\smallsetminus A_k$}\}.
  \]
  Then $\{\V(fg),\V(gf)\}$ is a partition of
  $\V$, and
   $g-f=1_{\V(gf)}-1_{\V(fg)}$.
\end{lemma}

\begin{figure}
    \begin{center}

      \begin{subfigure}{.35\textwidth}
        \centering
        \begin{tikzpicture}
          \begin{scope}
\clip (-1.125, -1.125) -- (3.375, -1.125) -- (3.375, 3.375) -- (-1.125, 3.375) -- cycle;
\draw[very thick] (-1.125, -0.75) -- (3.375, -0.75);
\draw[very thick] (-0.75, -1.125) -- (-0.75, 3.375);
\draw[very thick] (-1.125, 0.0) -- (3.375, 0.0);
\draw[very thick] (0.0, -1.125) -- (0.0, 3.375);
\draw[very thick] (-1.125, 0.75) -- (3.375, 0.75);
\draw[very thick] (0.75, -1.125) -- (0.75, 3.375);
\draw[very thick] (-1.125, 1.5) -- (3.375, 1.5);
\draw[very thick] (1.5, -1.125) -- (1.5, 3.375);
\draw[very thick] (-1.125, 2.25) -- (3.375, 2.25);
\draw[very thick] (2.25, -1.125) -- (2.25, 3.375);
\draw[very thick] (-1.125, 3.0) -- (3.375, 3.0);
\draw[very thick] (3.0, -1.125) -- (3.0, 3.375);
\draw[very thin] (-1.125, -0.375) -- (3.375, -0.375);
\draw[very thin] (-1.125, 2.625) -- (3.375, 2.625);
\draw[very thin] (2.625, -1.125) -- (2.625, 3.375);
\draw[very thin] (-0.375, -1.125) -- (-0.375, 3.375);
\draw[black!33, very thick] (-1.125, 1.125) -- (1.125, 1.125) -- (1.125, -1.125);
\draw[black!33, very thick] (1.125, 3.375) -- (1.125, 1.875) -- (1.875, 1.875) -- (1.875, 1.125) -- (3.375, 1.125);
\draw[black!66, very thick] (-1.125, 0.375) -- (0.375, 0.375) -- (0.375, -1.125);
\draw[black!66, very thick] (0.375, 3.375) -- (0.375, 1.125) -- (1.875, 1.125) -- (1.875, 0.375) -- (3.375, 0.375);
\draw[gray, fill=white] (-0.75, -0.75) circle (0.3);
\node at (-0.75, -0.75) {$+$};
\draw[gray, fill=white] (-0.75, 0.0) circle (0.3);
\node at (-0.75, 0.0) {$+$};
\draw[gray, fill=white] (-0.75, 0.75) circle (0.3);
\node at (-0.75, 0.75) {$-$};
\draw[gray, fill=white] (-0.75, 1.5) circle (0.3);
\node at (-0.75, 1.5) {$+$};
\draw[gray, fill=white] (-0.75, 2.25) circle (0.3);
\node at (-0.75, 2.25) {$+$};
\draw[gray, fill=white] (-0.75, 3.0) circle (0.3);
\node at (-0.75, 3.0) {$+$};
\draw[gray, fill=white] (0.0, -0.75) circle (0.3);
\node at (0.0, -0.75) {$+$};
\draw[black, fill=white] (0.0, 0.0) circle (0.3);
\node at (0.0, 0.0) {$+$};
\draw[black, fill=white] (0.0, 0.75) circle (0.3);
\node at (0.0, 0.75) {$-$};
\draw[black, fill=white] (0.0, 1.5) circle (0.3);
\node at (0.0, 1.5) {$+$};
\draw[black, fill=white] (0.0, 2.25) circle (0.3);
\node at (0.0, 2.25) {$+$};
\draw[gray, fill=white] (0.0, 3.0) circle (0.3);
\node at (0.0, 3.0) {$+$};
\draw[gray, fill=white] (0.75, -0.75) circle (0.3);
\node at (0.75, -0.75) {$-$};
\draw[black, fill=white] (0.75, 0.0) circle (0.3);
\node at (0.75, 0.0) {$-$};
\draw[black, fill=white] (0.75, 0.75) circle (0.3);
\node at (0.75, 0.75) {$-$};
\draw[black, fill=white] (0.75, 1.5) circle (0.3);
\node at (0.75, 1.5) {$-$};
\draw[black, fill=white] (0.75, 2.25) circle (0.3);
\node at (0.75, 2.25) {$-$};
\draw[gray, fill=white] (0.75, 3.0) circle (0.3);
\node at (0.75, 3.0) {$-$};
\draw[gray, fill=white] (1.5, -0.75) circle (0.3);
\node at (1.5, -0.75) {$+$};
\draw[black, fill=white] (1.5, 0.0) circle (0.3);
\node at (1.5, 0.0) {$+$};
\draw[black, fill=white] (1.5, 0.75) circle (0.3);
\node at (1.5, 0.75) {$+$};
\draw[black, fill=white] (1.5, 1.5) circle (0.3);
\node at (1.5, 1.5) {$-$};
\draw[black, fill=white] (1.5, 2.25) circle (0.3);
\node at (1.5, 2.25) {$+$};
\draw[gray, fill=white] (1.5, 3.0) circle (0.3);
\node at (1.5, 3.0) {$+$};
\draw[gray, fill=white] (2.25, -0.75) circle (0.3);
\node at (2.25, -0.75) {$+$};
\draw[black, fill=white] (2.25, 0.0) circle (0.3);
\node at (2.25, 0.0) {$+$};
\draw[black, fill=white] (2.25, 0.75) circle (0.3);
\node at (2.25, 0.75) {$-$};
\draw[black, fill=white] (2.25, 1.5) circle (0.3);
\node at (2.25, 1.5) {$+$};
\draw[black, fill=white] (2.25, 2.25) circle (0.3);
\node at (2.25, 2.25) {$+$};
\draw[gray, fill=white] (2.25, 3.0) circle (0.3);
\node at (2.25, 3.0) {$+$};
\draw[gray, fill=white] (3.0, -0.75) circle (0.3);
\node at (3.0, -0.75) {$+$};
\draw[gray, fill=white] (3.0, 0.0) circle (0.3);
\node at (3.0, 0.0) {$+$};
\draw[gray, fill=white] (3.0, 0.75) circle (0.3);
\node at (3.0, 0.75) {$-$};
\draw[gray, fill=white] (3.0, 1.5) circle (0.3);
\node at (3.0, 1.5) {$+$};
\draw[gray, fill=white] (3.0, 2.25) circle (0.3);
\node at (3.0, 2.25) {$+$};
\draw[gray, fill=white] (3.0, 3.0) circle (0.3);
\node at (3.0, 3.0) {$+$};
\end{scope}
        \end{tikzpicture}
        \subcaption{}
      \end{subfigure}
      \qquad
      \begin{subfigure}{.35\textwidth}
        \centering
        \begin{tikzpicture}
          \begin{scope}
\clip (-0.375, -1.125) -- (2.625, -1.125) -- (2.625, 3.375) -- (-0.375, 3.375) -- cycle;
\begin{scope}
\clip (-0.375, -0.375) -- (2.625, -0.375) -- (2.625, 2.625) -- (-0.375, 2.625) -- cycle;
\fill[RoyalBlue!50] (-1.125, -1.125) -- (3.375, -1.125) -- (3.375, 3.375) -- (-1.125, 3.375) -- cycle;
\draw[ultra thin, fill=RoyalBlue!25] (-1.125, 1.125) -- (1.125, 1.125) -- (1.125, -1.125) -- (0.375, -1.125) -- (0.375, 0.375) -- (-1.125, 0.375) -- cycle;
\draw[ultra thin, fill=RoyalBlue!25] (1.125, 3.375) -- (1.125, 1.875) -- (1.875, 1.875) -- (1.875, 1.125) -- (3.375, 1.125) -- (3.375, 0.375) -- (1.875, 0.375) -- (1.875, 1.125) -- (0.375, 1.125) -- (0.375, 3.375) -- cycle;
\end{scope}

\end{scope}

\draw (-0.375, -0.375) -- (2.625, -0.375) -- (2.625, 2.625) -- (-0.375, 2.625) -- cycle;
        \end{tikzpicture}
        \subcaption{}
      \end{subfigure}

      \caption
    {
      One the left: the difference $\tilde g-\tilde f$ for $f$ and $g$ as in Figure~\ref{fig_small_example}.
      The symbols $\pm$ represent a height difference of $\pm1$.
      Observe that the height difference is constant on each area enclosed by two grey lines.
      These areas correspond exactly to the sets $A_k\smallsetminus B_{k-1}$
      and $B_k\smallsetminus A_k$ in Lemma~\ref{awinbwanoafbw}.
      Remark that the grey lines of different colours may ``touch''
      each other, but that they cannot fully ``cross''.
      On the right: the sets $\mathbb T(gf),\mathbb T(fg)\subset\mathbb T$
      in dark blue and light blue respectively.
      (These sets are defined in Lemma~\ref{lemma:first_volume_lemma}.)
      These sets are obtained by replacing the $+$ and $-$ vertices on the left respectively by 
      unit hypercubes, and mapping the sets so obtained to the torus through 
      the map $\pi$.
    }
    \label{fig_small_neighbour_example}
    \end{center}
  \end{figure}

\begin{proof}%[Proof of Lemma~\ref{awinbwanoafbw}]
  First assume that $f\sim g$ for some $f\in V$ and $g\in W$.
  Then $\tilde g-\tilde f$ takes values in $\{-1,1\}$
  and therefore
  \[\cdots \subset \{\tilde f\geq 2\}\subset \{\tilde g\geq 1\}\subset \{\tilde f\geq 0\}\subset \{\tilde g\geq -1\}\subset \{\tilde f\geq -2\}\subset \cdots,\]
  that is, we may take $A_k=\{\tilde f\geq -2k\}$
  and $B_k=\{\tilde g\geq -2k-1\}$.

Now focus on the converse statement,
and suppose that the appropriate indexations
of $\{\{\tilde f\geq\lambda\}:\lambda\in\mathbb R\}$
and $\{\{\tilde g\geq\lambda\}:\lambda\in\mathbb R\}$
exist.
Then we may assume, by adding constants to $f$ and $g$,
that
$A_k=\{\tilde f\geq -2k\}$ with $f(\mathbf 0)$ even,
and that $B_k=\{\tilde g\geq -2k-1\}$ with $g(\mathbf 0)$ odd.
This implies that $\tilde f-\tilde g$ takes values in $\{-1,1\}$,
that is, $f$ and $g$ are neighbours.

In fact, if $\tilde g-\tilde f$ is not constant, then the same construction implies that
$\tilde g-\tilde f$ equals $1$ on the set
$\cup_k(B_k\smallsetminus A_k)$,
and $-1$
on the set $\cup_k (A_k\smallsetminus B_{k-1})$;
see Figure~\ref{fig_small_neighbour_example}.
This implies immediately that
$g-f=1_{\V(gf)}-1_{\V(fg)}$.
Since $g-f$ takes values in
$\{-1,1\}$, this also implies
that $\{\V(fg),\V(gf)\}$ is a partition of $\V$.
\end{proof}

\subsection{Irreducibility of the random walk}

The following lemma can be considered isolated from the rest of the analysis in this article,
but it is required to prove Theorem~\ref{main}.

\begin{lemma}
  The random walk on $\mathbf p\mathbf n$-periodic height functions is irreducible.
\end{lemma}

\begin{proof}
  We need to show that the graph of $\p\n$-functions is connected.
  If $f$ and $g$ differ by a constant, then
  $f$ and $g$ are clearly connected, since $f$ is a neighbour of $f+1$,
  which is a neighbour of $f+2$, et cetera.
  We need to show that any two shapes are connected.
  For any height function $f$, define
  \[\chi(f):=\sum\nolimits_{\x\in\V} (f(\x)-f(\0)).\]
  This function is constant on shapes,
  and there is a unique shape that minimises $\chi$;
  it is the unique shape $A^*=[f^*]$ such that
  \[\nabla f^*(\{\x,\x+\mathbf e_i\})=
\begin{cases}
  -1&\text{ if $0\leq \x_i <\n_i$,}\\
  1&\text{ if $\n_i\leq  \x_i <\t_i$.}
\end{cases}
  \]
  It is straightforward to check that the $\p\n$-function $f^*$ exists and minimises $\chi$.
  It suffices to find for every  $f\not\in A^*$
  a neighbour $g\sim f$ such that $\chi (g)<\chi(f)$.
  Note that $A^*$ is the unique shape such that $\0$ is the only local maximum
  of $f^*$ in the torus graph---where by local maximum we mean that all values that $f^*$
  takes on the neighbours of $\0$ are smaller than $f^*(\0)$.
  Let $f\not\in A^*$.
  Then $f$ must have some local maximum $\x\in \V\smallsetminus \{\0\}$ with respect to the torus graph.
  Then $g:=f+1-2\cdot 1_{\{\x\}}$ is a neighbour of $f$,
  and $\chi(g)=\chi(f)-2$.
\end{proof}

  % !TEX root=../ms.tex

\section{Embedding fractures in strips}
\label{sec:embedding}

Assume the setting of Lemma~\ref{awinbwanoafbw},
and recall that $\hat f$ denotes the average height of a $\p\n$-function $f$.
Lemma~\ref{awinbwanoafbw} implies that
the difference $\hat g-\hat f$ can be expressed
directly in terms of the sets $\V(fg)$ and $\V(gf)$.
More precisely, we have
$\hat g-\hat f=(|\V(gf)|-|\V(fg)|)/|\V|$.
This is the starting point
of a geometrical construction which
will eventually enable us to
approximate the process $\hat X$ by a martingale.

We shall do so as follows.
First, we replace the partition $\{\V(fg),\V(gf)\}$
by a partition of the unit torus $\mathbb T=\mathbb R^d/\mathbb Z^d$,
so that the cardinalities in the previous expression for $\hat g-\hat f$ 
can be replaced by the Lebesgue measures of these new sets.
These new sets are the natural blowups of $\V(gf)$ and $\V(fg)$,
appropriately scaled to fit the unit torus.
Second, we modify the new sets slightly so that the resulting process becomes a martingale.
We finally show that the modification does not affect the diffusivity in the limit.

The formal procedure of introducing the new partition is slightly technical,
although natural in spirit and illustrated by Figures~\ref{fig_small_neighbour_example} and~\ref{fig_STRIPS}.

\begin{definition}
Write $K=[-\frac12,\frac12]^d$ for a unit cube of $\mathbb R^d$ centred at $\0$.
If $A$ is a stepped surface
then we write $A+K:=\cup_{\x\in A}K+\x$;
this is the set $A$ with each vertex replaced
by a unit cube centred at that vertex.
Define $\partial A:=\partial(A+K)$,
the topological boundary of $A+K$.
The topological boundary $\partial A$ of a 
stepped surface $A$ associated to a height function
is represented by the grey lines on the right in 
Figure~\ref{fig_small_example}.
Recall the definition of the projection map $\pi$ from Subsection~\ref{subsec:notation},
and bring $\partial A$
to the unit torus $\T$ by 
defining $\partial_\T A:=\pi(\partial A)$.
Remark that the set $\partial_\T A$
is independent of the choice of $A\in F$
for a fixed fracture $F$,
and therefore we may rightfully write
$\partial_\T F:=\partial_\T A$.
\end{definition}

It is straightforward to work out from the geometrical picture (cf.~Figure~\ref{fig_small_example}) that $\partial A$
may be obtained from $A$ by
replacing each edge $\{\x,\x+\e_i\}$
that satisfies $\x\in A$ and $\x+\e_i\not\in A$,
with a unit hypercube of codimension one,
orthogonal to the edge $\{\x,\x+\e_i\}$,
and centred at $\x+\e_i/2$.
This implies in particular that the Hausdorff distance from $\partial^* A$
to $\partial A$ is bounded by $\sqrt d$.
Moreover, the set $\partial A$ is connected.

\begin{lemma}
  \label{lemma:first_volume_lemma}
Assume the setting of Lemma~\ref{awinbwanoafbw}.
If $f$ and $g$ are neighbours of different shape,
then define
\[
  \T(fg):=\cup_k\pi((A_k+K)\smallsetminus(B_{k-1}+K)),~
  \T(gf):=\cup_k\pi((B_k+K)\smallsetminus(A_k+K)).
\]
Then $\{\T(fg),\T(gf)\}$ is a partition of $\T$,
and $\hat g-\hat f=\operatorname{Vol}(\T(gf))-\operatorname{Vol}(\T(fg))$.
\end{lemma}

% !TEX root = ../ms.tex

\newcommand{\pluscolour}{RoyalBlue!50}
\newcommand{\minuscolour}{RoyalBlue!25}

\newcommand{\pluscolouralt}{Sepia!30}
\newcommand{\minuscolouralt}{Sepia!10}

% !TEX root = ../ms.tex

\newcommand{\alphaone}{
  (7.5,0)--
  (7.5,1.5)--
  (6.5,1.5)--
  (6.5,3.5)--
  (3.5,3.5)--
  (3.5,4.5)--
  (2.5,4.5)--
  (2.5,5.5)--
  (1.5,5.5)--
  (1.5,6.5)--
  (0,6.5)
}
\newcommand{\alphaonereverse}{
  (0,6.5)--
  (1.5,6.5)--
  (1.5,5.5)--
  (2.5,5.5)--
  (2.5,4.5)--
  (3.5,4.5)--
  (3.5,3.5)--
  (6.5,3.5)--
  (6.5,1.5)--
  (7.5,1.5)--
  (7.5,0)
}
\newcommand{\alphatwo}{
  (25,6.5)--
  (23.5,6.5)--
  (23.5,7.5)--
  (22.5,7.5)--
  (22.5,9.5)--
  (21.5,9.5)--
  (21.5,11.5)--
  (20.5,11.5)--
  (20.5,11.5)--
  (18.5,11.5)--
  (18.5,12.5)--
  (17.5,12.5)--
  (17.5,15.5)--
  (16.5,15.5)--
  (16.5,17.5)--
  (14.5,17.5)--
  (14.5,18.5)--
  (13.5,18.5)--
  (13.5,21.5)--
  (12.5,21.5)--
  (12.5,22.5)--
  (9.5,22.5)--
  (9.5,23.5)--
  (7.5,23.5)--
  (7.5,25)
}
\newcommand{\alphatworeverse}{
  (7.5,25)--
  (7.5,23.5)--
  (9.5,23.5)--
  (9.5,22.5)--
  (12.5,22.5)--
  (12.5,21.5)--
  (13.5,21.5)--
  (13.5,18.5)--
  (14.5,18.5)--
  (14.5,17.5)--
  (16.5,17.5)--
  (16.5,15.5)--
  (17.5,15.5)--
  (17.5,12.5)--
  (18.5,12.5)--
  (18.5,11.5)--
  (20.5,11.5)--
  (20.5,11.5)--
  (21.5,11.5)--
  (21.5,9.5)--
  (22.5,9.5)--
  (22.5,7.5)--
  (23.5,7.5)--
  (23.5,6.5)--
  (25,6.5)
}
\newcommand{\betaone}{
  (18.5,0)--
  (18.5,1.5)--
  (15.5,1.5)--
  (15.5,2.5)--
  (14.5,2.5)--
  (14.5,3.5)--
  (13.5,3.5)--
  (13.5,4.5)--
  (10.5,4.5)--
  (10.5,5.5)--
  (9.5,5.5)--
  (9.5,7.5)--
  (8.5,7.5)--
  (8.5,9.5)--
  (6.5,9.5)--
  (6.5,10.5)--
  (5.5,10.5)--
  (5.5,13.5)--
  (4.5,13.5)--
  (4.5,14.5)--
  (2.5,14.5)--
  (2.5,16.5)--
  (0,16.5)
}
\newcommand{\betatwo}{
  (25,16.5)--
  (23.5,16.5)--
  (23.5,17.5)--
  (22.5,17.5)--
  (22.5,20.5)--
  (21.5,20.5)--
  (21.5,21.5)--
  (19.5,21.5)--
  (19.5,22.5)--
  (18.5,22.5)--
  (18.5,25)
}
\newcommand{\betatworeverse}{
  (18.5,25)--
  (18.5,22.5)--
  (19.5,22.5)--
  (19.5,21.5)--
  (21.5,21.5)--
  (21.5,20.5)--
  (22.5,20.5)--
  (22.5,17.5)--
  (23.5,17.5)--
  (23.5,16.5)--
  (25,16.5)
}
\newcommand{\betabarone}{
  (0,24.5)--
  (0.5,24.5)--
  (0.5,23.5)--
  (1.5,23.5)--
  (1.5,20.5)--
  (2.5,20.5)--
  (2.5,19.5)--
  (4.5,19.5)--
  (4.5,18.5)--
  (5.5,18.5)--
  (5.5,14.5)--
  (8.5,14.5)--
  (8.5,13.5)--
  (9.5,13.5)--
  (9.5,12.5)--
  (10.5,12.5)--
  (10.5,11.5)--
  (13.5,11.5)--
  (13.5,10.5)--
  (14.5,10.5)--
  (14.5,8.5)--
  (15.5,8.5)--
  (15.5,6.5)--
  (17.5,6.5)--
  (17.5,5.5)--
  (18.5,5.5)--
  (18.5,2.5)--
  (19.5,2.5)--
  (19.5,1.5)--
  (21.5,1.5)--
  (21.5,0)
}
\newcommand{\betabartwo}{
  (21.5,25)--
  (21.5,24.5)--
  (25,24.5)
}

\begin{figure}
\begin{center}
\begin{tikzpicture}[x=1cm,y=1cm,scale=0.15]

\newcommand{\whitecirclelabel}[5]{
  \fill[color=white] (#1,#2) circle (#4);
  \node[#5] at (#1,#2) {#3};
}
\newcommand{\fillplus}[1]{
  \draw[fill=\pluscolour] #1;
}
\newcommand{\plussquare}{
  \color{\pluscolour}
  \blacksquare
  \color{black}
}
\newcommand{\fillminus}[1]{
  \draw[fill=\minuscolour]#1;
}
\newcommand{\minussquare}{
  \color{\minuscolour}
  \blacksquare
  \color{black}
}
\newcommand{\dota}[2]{
  \fill (#1,#2) circle [radius=0.25];
}

\begin{scope}[shift={(0,0)}]
  \begin{scope}
    \clip(0,0)rectangle(25,25);

    \fillminus{(0,0)--\alphaone--cycle}
    \fillplus{\alphaonereverse--\betaone--cycle}
    \fillminus{\betaone--(0,25)--\alphatworeverse--(25,0)--cycle}
    \fillplus{\betatwo--\alphatworeverse--cycle}
    \fillminus{\betatwo--(25,25)--cycle}

  \end{scope}
  \draw(0,0)rectangle(25,25);

  \node[anchor=south] at (12.5,25) {$\hat g-\hat f$};
  \node[anchor=north west] at (0,0) {$
  \begin{aligned}
    &\minussquare~\T(fg)\\
    &\plussquare~\T(gf)
  \end{aligned}
  $};

\end{scope}

\begin{scope}[shift={(27,0)}]
  \begin{scope}
    \clip(0,0)rectangle(25,25);

    \draw \alphaone;
    \draw \alphatwo;
    \draw \betaone;
    \draw \betatwo;

    \fillminus{(0,0)--(0,5)--(5,0)--cycle}
    \fillminus{(0,20)-- (0,25)--(5,25)--(25,5)--(25,0)--(20,0)-- cycle}
    \fillminus{(25,25)-- (25,20)--(20,25)-- cycle}
    \fillplus{(10,0)-- (15,0)--(0,15)--(0,10)-- cycle}
    \fillplus{(15,25)-- (10,25)--(25,10)--(25,15)-- cycle}

    \draw(10,0)-- (15,0)--(0,15)--(0,10)-- cycle;
    \draw(15,25)-- (10,25)--(25,10)--(25,15)-- cycle;

  \end{scope}
  \draw (0,0) -- (25,0) -- (25,25) -- (0,25) -- cycle;
  \node[anchor=south] at (12.5,25) {$\hat g+\kappa([g])-\hat f-\kappa([f])$};
  \node[anchor=north west] at (0,0) {$
  \begin{aligned}
    &\minussquare~\T(fg)\smallsetminus U\\
    &\plussquare~\T(gf)\smallsetminus U
  \end{aligned}
  $};

\end{scope}

\begin{scope}[shift={(54,0)}]
  \begin{scope}
    \clip(0,0)rectangle(25,25);

  \fillminus{(5,0)--\alphaone--(0,5)--cycle}
  \fillplus{\alphaone -- (0,10)--(10,0) -- cycle}

  \draw \betaone -- (0,15)--(15,0) -- cycle;
  \draw \betaone -- (0,20)--(20,0) -- cycle;

  \fillminus{(25,5)-- \alphatwo --(5,25)-- cycle}
  \fillplus{\alphatwo -- (10,25)-- (25,10)-- cycle}

  \draw \betatwo --(15,25) -- (25,15)-- cycle;
  \draw \betatwo--(20,25)  -- (25,20)-- cycle;

  \end{scope}
  \draw (0,0) -- (25,0) -- (25,25) -- (0,25) -- cycle;
  \node[anchor=south] at (12.5,25) {$\kappa([f])$};
  \node[anchor=north west] at (0,0) {$
  \begin{aligned}
    &\minussquare~\T(fg)\cap U_f\\
    &\plussquare~\T(gf)\cap U_f
  \end{aligned}
  $};

\end{scope}

\node at (4,20) [anchor=south] {$\partial_\T F_1$};
\draw [->] (4,20) to [bend left=5] (1,7);
\draw [->] (4,20) to [bend right=5] (13,20.5);
\node at (21,5) [anchor=north] {$\partial_\T G_1$};
\draw [->] (21,5) to [bend right=5] (10,6.5);
\draw [->] (21,5) to [bend left=5] (24,16);
% \draw [->] (4,22.2) to [bend left=5] (7,23.5);

%
% \node at (10,28) {$D$};
% \draw [-] (9.5,26.5) to [bend left=5] (6,7);
% \dota{6}{7}
% \draw [-] (10.5,26.5) to [bend right=10] (11,24);
% \dota{11}{24}
%
% \node at (15,28) {$C$};
% \draw [-] (14.5,26.5) to [bend left=5] (10,11);
% \dota{10}{11}
% \draw [-] (15.5,26.5) to [bend right=10] (20,24);
% \dota{20}{24}
%
% \node at (22,28) {$\partial^\nu G_1$};
% \draw [->] (21.5,26.5) to [bend left=5] (13,4.75);
% \draw [->] (22.5,26.5) to [bend left=5] (23.5,17.75);

% \node at (40,28) {$C-U$};
% \node at (49,28) {$D-U$};

% \node at (67,28) {$D\cap U_F$};

\end{tikzpicture}

\caption{
For each subfigure, the value of its label equals $\operatorname{Vol}(\color{\pluscolour}
\blacksquare
\color{black})-\operatorname{Vol}(\color{\minuscolour}
\blacksquare
\color{black})$.}
\label{fig_STRIPS}
\end{center}
\end{figure}

See Figure~\ref{fig_small_neighbour_example} for the construction of the sets
$\T(fg)$ and $\T(gf)$. 
See the leftmost subfigure in Figure~\ref{fig_STRIPS}
for the equation in the lemma.

\begin{proof}[Proof of Lemma~\ref{lemma:first_volume_lemma}]
  It is a simple exercise to work out that the sets $\T(fg)$ and $\T(gf)$
  are obtained from $\V(fg)$ and $\V(gf)$ respectively as follows:
  first replace each vertex $\x\in \V$
  by the cube $K+\x\subset \R^d/\sigma\Z^d$,
  then apply the map $\sigma^{-1}$.
  The volume of each cube $K$
  equals one;
  the map $\sigma^{-1}$ has determinant $1/|\V|$.
\end{proof}

With the geometric picture at hand, we can start reducing to the one-dimensional model that was mentioned in the introduction.
The immediate goal is to move from the leftmost picture in Figure~\ref{fig_STRIPS}
to the rightmost picture.
Essentially, the set $U_f$ in the rightmost picture is obtained from 
$\partial_\mathbb TF_1$ by associating to $\partial_\mathbb TF_1$ the smallest
\emph{cylinder}, where a cylinder is a subset of the torus invariant under certain 
translations.
The natural way to obtain such a set is by first projecting the set to a smaller space,
then taking the preimage of the projection.
We shall perform this operation in the universal cover rather than 
in the torus so that we may work with linear maps.
To this end, we introduce the linear projection map $P$ which projects $\R^d$
onto $\R$ along the subspace $\sigma\n^\perp$.

\begin{definition}
  Write $P:\R^d\to\R$ for the unique linear map
  such that $P(\e_1)=1$ and $\operatorname{Ker}P=\sigma\n^\perp$.
  If $A$ is a stepped surface, then define
  the \emph{minimal strip}
  $U(A)$ of $A$ by
  \[
    U(A):=P^{-1}P\partial A=\partial A+\sigma\n^\perp\subset\R^d.
  \]
  Note that $U(A)$ is the smallest
  $\sigma\n^\perp$-invariant subset of $\mathbb R^d$
  containing $\partial A$.
\end{definition}

Consider a stepped surface $A$.
As $\partial A$ is connected,
the set $P\partial A$ is an interval.
If $A$ is $\sigma N$-invariant,
then $\partial A/\sigma N$
is a compact subset of $\R^d/\sigma N$,
which implies compactness of the interval $P\partial A$.
Now let $B$ denote another $\sigma N$-invariant stepped
surface.
Suppose that $U(B)$ is disjoint from $U(A)$,
say with
$
  P\partial A<
  P\partial B
$
without loss of generality,
where we write $a<b$ for a pair
$(a,b)$ of intervals whenever $\sup a<\inf b$.
Then it is straightforward to work out that the original stepped surfaces 
are also ordered in the sense that $A\sqsubset B$.
Thus, in general, if $U(A)$ and $U(B)$
are disjoint, then either $A\sqsubset B$, or $B\sqsubset A$,
depending on the ordering of the intervals $P\partial A$
and $P\partial B$.

Let us now translate the previous definitions
to the unit torus $\T$ by application of the map $\pi$.
Write $\S$ for $\R/(|\n|/\n_1)\Z$.
Note that the hyperplane $\n^\perp+\Z^d$ and the circle $\R\e_1+\Z^d$
intersect at exactly $\n_1/|\n|$ points,
as subsets of $\T$.
This implies that for each $\x\in\T$,
there is a unique element $x\in\S$
such that $\x\in x\e_1+\n^\perp+\Z^d$.
% In other words, the map $P_\T$ in
% the following definition is well-defined.
This motivates the equations in the following definition.

\begin{definition}
  Write $P_\T:\T\to\S$
  for the unique map
  such that $\x\in P_\T(\x)\e_1+\n^\perp+\Z^d$
  for any $\x\in\T$.
  If $A$ is a $\sigma N$-invariant stepped surface,
  then define the \emph{minimal strip}
  $U_\T(A)$
  by
  \[
    U_\T(A):=\pi U(A)
    =P_\T^{-1}P_\T\partial_\T A
    =\partial_\T A+\n^\perp
    \subset\T.
  \]
  This is the smallest $\n^\perp$-invariant
  subset of $\T$
  which contains $\partial_\T A$.
  Finally, if $F\in F(\p,\n)$,
  then define
  $U_\T(F):=U_\T(A)$
  where $A\in F$;
  this definition is independent of the choice of $A\in F$.
\end{definition}

See the rightmost subfigure of Figure~\ref{fig_STRIPS} for an illustration of a minimal strip of a fracture.
By reasoning as before, we observe that $P_\T\partial_\T F$ is
a compact connected subset of the circle $\S$.
The interval $P_\T\partial_\T F$ and
the strip $U_\T(F)$ are thus characterised entirely
by the endpoints of this interval,
unless $P_\T\partial_\T F=\S$,
in which case $U_\T(F)=\T$.

\begin{definition}
  Define
  $h:F(\mathbf p,\mathbf n)\to \S$
  and $r:F(\mathbf p,\mathbf n)\to[0,|\mathbf n|/2\mathbf n_1]$
  such that
  \[
    P_\T\partial_\T F=h(F)+[-r(F),r(F)]
  \]
  for any $F\in F(\mathbf p,\mathbf n)$.
  If $P_\T\partial_\T F\neq\S$
  then this equality uniquely defines $h(F)$
  and $r(F)$; otherwise we set $h(F):=0$
  and $r(F):=|\mathbf n|/2\mathbf n_1$.
  We shall also write $[h(F)]$ for the unique number in $[0,|\mathbf n|/\mathbf n_1)$
  such that $h(F)=[h(F)]+(|\mathbf n|/\mathbf n_1)\mathbb Z$.
\end{definition}

If $F$ is a fracture and $P_\T\partial_\T F\neq \S$,
then the collection of minimal cylinders $(U(A))_{A\in F}$
is pairwise disjoint.
In particular, our previous observation now implies that $\sqsubset$
restricts to a total order on the stepped surfaces in $F$.

\begin{definition}
  \label{definition_of_D_nogwat}
  For fixed $n\in\mathbb N$,
  define
  \begin{align*}
    D(\mathbf p,\mathbf n,n)&:=
    \{
      \text{the strips $(U_\T(F_k))_{1\leq k\leq n}$ are pairwise disjoint}
    \}
    \\&\phantom{:}=
    \{
      \text{the intervals $(P_\T\partial_\T F_k)_{1\leq k\leq n}$ are pairwise disjoint}
    \}
    \subset F(\mathbf p,\mathbf n)^n.
  \end{align*}
\end{definition}

If $n\geq 2$ and $(F_1,\dots,F_n)\in D(\mathbf p,\mathbf n,n)$,
then all intervals $P_\T\partial_\T F_k$
are strict subsets of $\S$.
This means that all minimal cylinders in $(U(A))_{A\in\cup_kF_k}$
are pairwise disjoint, and consequently
$\sqsubset $ restricts to a total order on $\cup_kF_k$.
This implies the following lemma.

\begin{lemma}
  \label{lemma_yolo}
Let $n\geq 2$,
 let $(F_1,...,F_n)\in D(\mathbf p,\mathbf n,n)$,
 and write
 $h_k:=[h(F_k)]$.
Suppose that
$h_i<h_j$ whenever $i<j$.
Then we may pick a representative $A_k\in F_k$
of each fracture such that
\[\dots\sqsubset\theta^{-1}A_n\sqsubset A_1 \sqsubset A_2\sqsubset \dots\sqsubset A_n\sqsubset \theta A_1\sqsubset \theta A_2\sqsubset \cdots.\]
Make this into a $\mathbb Z$-indexed chain by setting
$A_{i+nj}=\theta^j A_i$ for $1\leq i\leq n$
and $j\in\mathbb Z$.
Then the strips $(U(A_k))_{k\in\mathbb Z}$
are pairwise disjoint.
\end{lemma}

% \begin{proof}
% Assume the setup of the lemma.
% As $n\geq 2$, this means
% that each set $P_\T\partial_\T F_k$
% is a strict subset of $\S$.
% In particular, this means
% that the cylinders in $(U(A))_{A\in\cup_k F_k}$
% are pairwise disjoint,
% and therefore $\sqsubset$ restricts
% to a total order on $(A)_{A\in\cup_kF_k}$.
% It is now straightforward to deduce the statements in the lemma.
% \end{proof}

Note in particular that the previous lemma relates
the ordering of the stepped surfaces directly to the ordering
of the numbers $(h_k)_{1\leq k\leq n}$.
The purpose of the introduction of cylinders, is the following lemma.
The lemma gives a much simpler expression for the height difference $\hat g-\hat f$,
although it requires a small correction which we can later show is negligible in the limit.
The lemma essentially says that---up to the correction term---the value 
of $\hat g-\hat f$ can be measured as the volume difference 
in the centre picture in  Figure~\ref{fig_STRIPS}
after introducing the correction term.
This is much simpler, because measuring the coloured areas in the centre 
picture equates to measuring the widths of the strips,
without taking into account the microscopic shape of the stepped surfaces.

\begin{lemma}
  \label{correction_lemma}
Suppose that $(F_1,\dots,F_{|\mathbf n|},G_1,\dots,G_{|\mathbf n|})\in D(\mathbf p,\n,2|\n|)$,
and write $h_k:=[h(F_k)]$
and $h_k':=[h(G_k)]$.
Then the shapes $V$ and $W$ corresponding to
$(F_k)_k$ and $(G_k)_k$ are neighbours if
and only if the indices of $(h_k)_k$
and those of $(h_k')_k$ can be reordered
such that either
\begin{equation}
  \label{eq:awdoiundaiwon}
  h_1<h_1'<\dots<h_{|\n|}<h_{|\n|}'
  \qquad\text{or}\qquad
  h_1'<h_1<\dots<h_{|\n|}'<h_{|\n|}.
\end{equation}

Moreover, there exists a function $\kappa:S(\p,\n)\to[-1,1]$
such that
\[
  \hat g+\kappa([g])-\hat f-\kappa([f])=
  1_{h_1'<h_1}-1_{h_1'>h_1}+2\frac{\n_1}{|\n|}\sum_{k=1}^{|\n|}h_k'-h_k
\]
whenever $f$ and $g$ are neighbouring height functions
with the corresponding fractures $(F_k)_k$
and $(G_k)_k$ satisfying $(F_1,\dots,F_{|\mathbf n|},G_1,\dots,G_{|\mathbf n|})\in D(\mathbf p,\n,2|\n|)$.
\end{lemma}

\begin{proof}
  For the first statement in this lemma, combine the first part of Lemma~\ref{awinbwanoafbw} with Lemma~\ref{lemma_yolo}.
  For the second statement in this lemma,
  we rely on Lemma~\ref{lemma:first_volume_lemma}
  and the intuitive picture of Figure~\ref{fig_STRIPS}.
  Select two neighbours $f$ and $g$
  such that the  $2|\n|$-tuple
  containing the corresponding fractures is
  contained in $D(\p,\n,2|\n|)$;
  see Figure~\ref{fig_STRIPS}.
  Define $\T(fg)$ and $\T(gf)$ as in the statement of Lemma~\ref{lemma:first_volume_lemma},
  and define the following subsets of $\T$:
  \[
    U_f:=\cup_kU_\T(F_k),
    \qquad
    U_g:=\cup_kU_\T(G_k),
    \qquad
    U:=U_f\cup U_g.
  \]
  Assume for now the claim
  that
  \begin{equation}
    \label{eq_to_prove_strip_width}
    \operatorname{Vol}(\T(gf)\smallsetminus U)-\operatorname{Vol}(\T(fg)\smallsetminus U)=
    1_{h_1'<h_1}-1_{h_1'>h_1}+2\frac{\n_1}{|\n|}\sum_{k=1}^{|\n|}h_k'-h_k.
  \end{equation}
  We must therefore find a function $\kappa:S(\p,\n)\to[-1,1]$
  such that
  \begin{align*}
    \kappa([g])-\kappa([f])&=
    \operatorname{Vol}(\T(fg)\cap U)-
    \operatorname{Vol}(\T(gf)\cap U)
    \\&=
    \operatorname{Vol}(\T(fg)\cap U_g)
    -\operatorname{Vol}(\T(gf)\cap U_g)\\
    &\qquad\qquad
    +\operatorname{Vol}(\T(fg)\cap U_f)
    -\operatorname{Vol}(\T(gf)\cap U_f).
  \end{align*}
  The following observation is crucial:
  consider $f$ fixed and
  $g$ variable, but
  conditional on $g\sim f$ and
  $(F_1,\dots,F_{|\n|},G_1,\dots,G_{|\n|})\in D(\p,\n,2|\n|)$.
Then the sets $\T(fg)\cap U_f$
and $\T(gf)\cap U_f$ are independent of $g$---see Figure~\ref{fig_STRIPS}.
We therefore simply define
$\kappa([f])=\operatorname{Vol}(\T(gf)\cap U_f)
-\operatorname{Vol}(\T(fg)\cap U_f)$ to obtain the final result;
the same definition implies that
$\kappa([g])=\operatorname{Vol}(\T(fg)\cap U_g)
-\operatorname{Vol}(\T(gf)\cap U_g)$
by symmetry.
If no neighbour $g$ of $f$ exists for the fixed function $f$ such that
the associated $2|\n|$-tuple belongs to $D(\p,\n,2|\n|)$,
then simply set $\kappa([f]):=0$.

We finally focus on the claim: Equation~\ref{eq_to_prove_strip_width}.
It suffices to consider the case $h_1'>h_1$;
the other case follows automatically by symmetry.
Write $r_k:=r(F_k)$ and $r_k':=r(G_k)$.
It is straightforward to work out that
\[
  \operatorname{Vol}(\T(gf)\smallsetminus U)
  =
  \frac{\n_1}{|\n|}\sum_{k=1}^{|\n|}
  h_k'-r_k'-h_k-r_k;
\]
each term corresponds to the width of one of the $|\n|$
connected components of the set $\operatorname{Vol}(\T(gf)\smallsetminus U)$.
Indeed, one observes in Figure~\ref{fig_STRIPS_2}
that the width of the dark blue strip equals $h_1'-r_1'-h_1-r_1$; in this figure $|\n|=\n_1=1$.
Similarly, for $\operatorname{Vol}(\T(fg)\smallsetminus U)$ we have
\[
\operatorname{Vol}(\T(fg)\smallsetminus U)
=
\frac{\n_1}{|\n|}\sum_{k=1}^{|\n|}
\begin{cases}
  h_{k+1}-r_{k+1}-h_k'-r_k'
&\text{if $k<|\n|$,}\\
(h_1+\frac{|\n|}{\n_1})-r_1-h_{|\n|}'-r_{|\n|}'
&\text{if $k=|\n|$.}
\end{cases}
\]
Indeed, the width of the light blue strip in Figure~\ref{fig_STRIPS_2} equals
$(h_1+1)-r_1-h_1'-r_1'$.
The last term is different owing to the cyclic nature of the line
$\mathbb R\e_1+\mathbb Z^d$ as a subset of $\T$.
A combination of the previous two displays yields Equation~\ref{eq_to_prove_strip_width}.
\end{proof}

% !TEX root = ../ms.tex

% \input{parts/partiii/figures/integrals_data}
\begin{figure}
\begin{center}
\begin{tikzpicture}[x=1cm,y=1cm,scale=0.30]

\begin{scope}
\clip (0,0) rectangle (25,25);
  % alpha and its minimal strip

  \fill[\minuscolour] (0,0)--(5,0) -- (0,5) --cycle;
  \fill[\pluscolour] (15,0) -- (10,0) -- (4,6) -- (9,6)--cycle;
  \fill[\minuscolour] (25,0) -- (20,0) -- (14,6) -- (24,6)--(25,5)--cycle;

  \fill[\pluscolouralt] (0,10) --(0,20)--(10,10);
  \fill[\pluscolouralt] (25,10) --(10,25)--(20,25)--(25,20);
  \fill[\minuscolouralt] (15,10) --(0,25)--(5,25)--(20,10);

  \draw\alphaone;
  \draw\alphatwo;
  \draw (5,0) -- (0,5);
  \draw (5,25) -- (25,5);
  \draw (10,0) -- (0,10);
  \draw (10,25) -- (25,10);

  % boundary points
  \fill[black] (6.5,3.5) circle (0.2);
  \node[anchor= south] at (6.5,3.5) {$\pi\x$};
  \fill[black] (17.5,12.5) circle (0.2);
  \node[anchor=north] at (17.5,12.5) {$\pi\y$};

  % rotation axis
  \fill[black] (12,8) circle (0.2);
  \node[anchor=south] at (12,8) {$\mathbf{m}$};

  % beta
  \draw\betaone;
  \draw\betatwo;
  \draw (15,0) -- (8,7);
  \draw (20,0) -- (13,7);
  % \draw [decorate,decoration={brace,amplitude=4pt}](3.5,6.5) -- (8.5,6.5);
  % \node[anchor=south] at (6,6.5) {$x_1$};
  % \draw [decorate,decoration={brace,amplitude=2pt}](13.5,6.5) -- (23.5,6.5);
  % \node[anchor=south] at (18.5,6.5) {$y_1$};

  % bar beta
  \draw[dashed] \betabarone;
  \draw[dashed] \betabartwo;
  \draw[dashed]  (11,9)--(0,20);
  \draw[dashed]  (16,9)--(0,25);
  \draw[dashed]  (20,25)--(25,20);

  % \draw [color=RoyalBlue,decorate,decoration={brace,amplitude=2pt,aspect=0.3}](0,10) -- (10,10);
  % \node[anchor=south,color=RoyalBlue] at (3,10) {$\bar x_1$};
  % \draw [color=RoyalBlue,decorate,decoration={brace,amplitude=2pt,aspect=0.3}](15,10) -- (20,10);
  % \node[anchor=south,color=RoyalBlue] at (16.5,10) {$\bar y_1$};

\end{scope}
\draw (0,0) rectangle (25,25);

% BOTTOM BRACES
\draw [decorate,decoration={brace,amplitude=5pt}](0,0) -- (7.5,0);
\node[anchor=south] at (3.75,0.5) {$h_1$};

\draw [decorate,decoration={brace,amplitude=5pt}](7.5,0) -- (10,0);
\node[anchor=south] at (8.75,0.5) {$r_1$};

\draw [decorate,decoration={brace,amplitude=5pt}](20,0) -- (17.5,0);
\node[anchor=north] at (18.75,-0.5) {$r_1'$};

\draw [decorate,decoration={brace,amplitude=5pt}](17.5,0) -- (0,0);
\node[anchor=north] at (8.75,-0.5) {$h_1'$};

% LINE INDICATORS
\node at (-3,20) [anchor=south] {$\partial_\T F_1$};
\draw [->] (-3,20) to [bend left=5] (1,7);
\draw [->] (-3,20) to [bend right=5] (13,20.5);

\node at (28,5) [anchor=north] {$\partial_\T G_1$};
\draw [->] (28,5) to [bend right=5] (10,6.5);
\draw [->] (28,5) to [bend left=5] (24,16);

\node at (28,20) [anchor=south] {$\partial_\T\bar G_1$};
\draw [->,dashed] (28,20) to [bend left=5] (6,17);
% \draw [->] (28,22) to [bend right=30] (25.5,24.5);

\end{tikzpicture}
\caption{Figure~\ref{fig_STRIPS} in more detail}
\label{fig_STRIPS_2}

\end{center}
\end{figure}

  % !TEX root = ../ms.tex

\section{Random stepped surfaces}
\label{sec:random}

The goal of this section is to prove the following theorem
by applying ideas from~\cite{SELF}.
Note that if in the following theorem we had additionally 
asked that $\mathbf p\to\infty$
with $\p_i/\p_j$ tending to some constant $c_{ij}$
for each $i$ and $j$, then the same result would have followed more easily
from the older and more general work \emph{Random surfaces}
of Sheffield~\cite{SHEFFIELD}.
The theorem asserts two things:
that random fractures are flat in the limit of $\p$,
and that the random variable $h(F)$
converges to the natural uniform distribution.
The first assertion is essentially a concentration result for random paths and 
random surfaces,
for which we refer to the existing literature and give a brief argument 
to tailor the results to the current setting.
The second assertion is a consequence of symmetry under shifts with the 
mesh size of the scaled lattice tending to zero.

\begin{theorem}
  \label{thm_asymptotic_h_r}
  The set $F(\p,\n)$ is finite.
  Write $\mathbb P_\p$ for the uniform probability measure on
  $F(\p,\n)$, and write $F$ for the random fracture in this probability measure.
  Then $r(F)\to 0$ in probability,
  and the distribution of $h(F)$
  converges weakly to the uniform continuous distribution on $\S$
  as $\mathbf p\to\infty$.
\end{theorem}

Recalling the relation between $r$ and $h$ and the minimal cylinders in the previous 
subsection, we observe that in the limit it is extremely unlikely that the minimal
cylinders of two randomly chosen fractures intersect.
The formal statement, which plays a key role later, is as follows.

\begin{corollary}
  \label{corolor}
  Fix $k\in\N$
  and write $\P_\p$
  for the uniform probability measure on
  $F(\p,\n)^k$.
  Then the event $D(\p,\n,k)$ has high probability in $\P_\p$ as $\mathbf p\to\infty$.
\end{corollary}

In this section, we are after a concentration result;
for Theorem~\ref{thm_asymptotic_h_r},
the hard part is to prove that $r(F)\to 0$ in probability,
the weak convergence of $h(F)$ to the correct distribution then
follows from symmetry arguments.
We introduce some ideas from~\cite{SELF}
so that we are able to quickly derive the desired concentration.
Random stepped surfaces are interesting objects in their own right,
and are extensively discussed in~\cite{SELF}.
For a detailed treatment of the subject, we refer to that article.

Before proceeding, we shall introduce an intuitive auxiliary object 
that simplifies the study of random fractures.

\begin{definition}
  Define $\mathbf u:=\mathbf e_1+\dots+\mathbf e_d$.
  Write $F_0(\p,\n)$
  for the set of
  $\sigma N$-invariant stepped surfaces $A$
  with the property that
  $\mathbf 0\in A$
  and $\mathbf u\not\in A$.
  This is equivalent to asking that $\u/2\in\partial A$.
  Elements in $F_0(\p,\n)$
  are called \emph{anchored fractures}.
\end{definition}

Anchored fractures are somewhat easier to work with than fractures:
an anchored fracture consists of a single stepped surface,
which is furthermore known to contain a specific point (the anchor) in its boundary.
The extra control simplifies some of the technical details that would otherwise appear.
A uniformly random fracture is obtained by first selecting a uniformly random anchored fracture,
then shifting the anchor to a random vertex in the torus.
This statement is the subject of the following proposition.

\begin{proposition}
  \label{propo_n_to_one}
  The map
  \[
    F_0(\p,\n)\times\prod\nolimits_{i=1}^d \{1,2,\dots,\s_i\}\to F(\p,\n),\,
    (A,\mathbf x)\mapsto
    \{\theta^k(A+\mathbf x):k\in\mathbb Z\}
  \]
  is $n$-to-$1$ for some $n\in\mathbb N$.
\end{proposition}

\begin{proof}
It follows from the definitions of fractures and anchored fractures
that the map is 
well-defined.
We shall calculate the cardinality $n$ of 
the preimage of some fracture $F$ explicitly,
and show that this number is independent of the choice of $F$.
By considering the local combinatorial structure of a stepped surface,
one can work out that
\[
  n = |\{(A,\x,i):A\in F,\,\x\in A,\,\x+\e_i\not\in A\}|,
\]
and we aim to calculate the cardinality of the set on the right.
For fixed $i$,
it makes sense to consider all $\y$ which belong to the same cycle 
$\{\x+k\e_i:k\in\mathbb Z\}\subset\mathbb V$,
as we did in Subsection~\ref{whatever}.
One can show that one can find exactly $\n_i/|\n|$ pairs $(A,\y)$
such that $\y$ belongs to this cycle and such that $(A,\y,i)$ belongs to the set whose cardinality we are after.
Thus, we conclude that
\[
  n=\sum_i \frac{\n_i}{|\n|}\prod_{j\neq i}\s_j.
  \]
% Fix $F$, and let $n$ denote the cardinality of the preimage of $F$.
% Note that $n$
% equals the number of pairs $(A,\x)$ such 
% that all of the following three statements hold true:
% \begin{enumerate}
% \item $A\in F$,
% \item $\x\in A$ and $\x+\u\not\in A$,
% \item $\x\in \prod\nolimits_{i=1}^d \{1,2,\dots,\s_i\}$.
% \end{enumerate}
% Such a pair $(A,\x)$ is called an \emph{anchor point} for $F$,
% and clearly corresponds to the element 
% $(A-\x,\x)$ in the preimage of $F$.
% We must show that each fracture has the same number of anchor points.
% 
\end{proof}

The concentration result is effectively contained in the following lemma.

\begin{lemma}
  \label{lemma_desired_concentration}
  The set $F_0(\p,\n)$ is finite.
  Write $\mathbb P_\p$ for the uniform probability measure on
  $F_0(\p,\n)$, and write $A$ for the random anchored fracture in this probability measure.
  Then $\sigma^{-1}\partial A$ converges to $\n^\perp$  in probability in the Hausdorff metric
  in $\mathbb P_\p$
  as $\p\to\infty$.
\end{lemma}

Let us first derive Theorem~\ref{thm_asymptotic_h_r}
from this lemma, before proving it.

% Recall that $\n$ and $\u$ are fixed vectors with positive coordinates,
% and that $\sigma$ is the linear automorphism of $\R^d$
% that maps $\e_i$ to $\s_i\e_i$.
% This means that $||\sigma^{-1}||=1/\min_i\s_i=1/\min_i(\p_i+\n_i)\leq 1/\min_i\p_i=o(1)$
% as $\p\to\infty$.

\begin{proof}[Proof of Theorem~\ref{thm_asymptotic_h_r}]
  Work in the setting of Proposition~\ref{propo_n_to_one}.
  Write $(A,\mathbf x)$
  for the random pair in the uniform probability
  measure $\mathbb P_\mathbf p$ on \[  F_0(\p,\n)\times\prod\nolimits_{i=1}^d \{1,2,\dots,\s_i\}
.\]
Then the distribution of
$\partial_\T (A+\mathbf x)$
in $\mathbb P_\mathbf p$
equals that of $\partial_\T F$
in the uniform probability measure on $F(\mathbf p,\mathbf n)$, due to the proposition.

Write $B_\varepsilon$ for a Euclidean ball in $\mathbb R^d$
of radius $\varepsilon$ centred around the origin.
For fixed $\varepsilon>0$,
Lemma~\ref{lemma_desired_concentration} says that
$\partial_\T(A+\mathbf x)\subset \mathbf n^\perp+\pi\mathbf x+B_\varepsilon+\Z^d$
with high probability as $\mathbf p\to\infty$,
and therefore $r(F)\to 0$ in probability
as $\mathbf p\to\infty$ as stated in the lemma.

Now conditional on $P_\T\partial_\T F\neq \S$,
that is, on $r(F)\neq |\mathbf n|/2\mathbf n_1$,
the distribution of $h(F)$
is invariant under shifts by $1/\mathbf s_1$.
But the conditioning event has high probability as $\mathbf p\to\infty$,
and the invariance mentioned implies that
$h(F)$ converges weakly to the uniform distribution,
since $\mathbf s_1=\mathbf p_1+\mathbf n_1\to\infty$ as $\mathbf p\to\infty$.
\end{proof}

\begin{proof}[Proof of Lemma~\ref{lemma_desired_concentration}]
  % !TEX root = ../ms.tex

The lemma is interesting in its own right and highly nontrivial.
The key ingredient is a concentration inequality in~\cite{SELF}.
The proof requires the introduction of some machinery of that paper,
so that we are able to efficiently derive the lemma
from the concentration inequality.

Informally, the proof runs as follows.
We first replace $\partial A$ by another notion of boundary,
which is sufficiently close to it in terms of Hausdorff distance.
This new notion of boundary has a natural parametrisation when $A$ 
is a stepped surface.
In particular, we may think of this new boundary as a random discrete function:
a \emph{height function}.
In $d=2$, this height function actually behaves like a simple random 
walk on $\mathbb Z$,
and in $d=3$ the height function corresponds exactly to the height function
of the dimer model on the hexagonal lattice.
For the concentration result that we are after, we essentially make three observations.
\begin{enumerate}
  \item We note that, after scaling by $\sigma^{-1}$, we are in fact dealing with Lipschitz functions on the compact space $\n^\perp/N$.
  This means in particular that \emph{pointwise concentration} of the surface implies uniform concentration 
  (that is, concentration of the surface around $\n^\perp$ in the Hausdorff distance).
  \item We demonstrate that the \emph{variance} of the height function at each point is small (after rescaling),
  so that the corresponding point on the random surface concentrates around some point in the rescaled space.
  \item We quote a simple result which asserts that the \emph{expectation} of the height function at each vertex lands 
  exactly (after rescaling) in $\n^\perp$, so that each point on the random surface concentrates 
  around a point in $\n^\perp$.
  Although we do not provide a proof, we mention that the expectation 
  can be calculated exactly due to a simple symmetry argument in our periodic setting.
\end{enumerate}
Jointly, these three steps prove the asserted result.

Let $A$ denote any stepped surface for now.
Write $\partial'A$ for the set of vertices $\x\in A$
such that $\x+\u\not\in A$.
The set $\partial'A$ should be thought of as an alternative
for $\partial A$.
Indeed,
the Hausdorff distance from $\partial A$ to $\partial' A$ is bounded
by $\sqrt d$,
hence
the Hausdorff distance from $\sigma^{-1}\partial A$ to $\sigma^{-1}\partial' A$
is bounded by $\|\sigma^{-1}\|\sqrt d=o(1)$ as $\p\to\infty$.
It therefore suffices to prove the lemma
with $\partial A$ replaced with $\partial'A$.

Let us focus on the set $\partial'A$;
we now cite a number of constructions
from Section~3 in~\cite{SELF}.
For any $\x\in\Z^d$, the line $\x+\R\u$ intersects
$\partial'A$ at exactly one vertex.
This suggests the following parametrisation for $\partial'A$.
Write $P':\R^d\to\u^\perp$ for orthogonal projection,
and $X^*$ for the image of $\Z^d$ under $P'$.
For any $\x\in X^*$, write $\x_A$ for the unique point of intersection of
$\x+\R\u$ with $\partial'A$.
Define the map $\phi_A:X^*\to\R$ by
$\phi_A(\x):=(\x_A,\u)$.
This means that
$\x_A=\x+\phi_A(\x)\mathbf u/d\in\partial'A$,
and all vertices of $\partial'A$ are reached by ranging
$\x$ over $X^*$.
The above construction gives a bijection
from the set of stepped surfaces to a suitable set of functions
on $X^*$.

Functions have more structure than sets, which is the merit of this construction.
If $A\in F_0(\p,\n)$, then $\0\in\partial'A$ by definition,
and $\sigma N\subset\partial'A$ because $A$ is $\sigma N$-invariant.
The $\sigma N$-invariance of stepped surfaces in $F_0(\p,\n)$
corresponds to the periodic setting as described in
Section~9 in~\cite{SELF}.
That section asserts that the set $F_0(\p,\n)$ is finite.
The rest of the proof depends on three observations,
one of which addresses a property of each individual stepped surface,
while the other two describe the distribution of
the random function $\phi_A$ in the measure $\P_\p$.
Write $<$ for the strict partial order on $\R^d$
such that $\x<\y$ if and only if $\x_i<\y_i$
for all $i$.
Write $\E_\p$ and $\Var_\p$ for the expectation and variance
of a random variable with respect to $\P_\p$,
the uniform probability measure on $F_0(\p,\n)$.

\begin{enumerate}

  \item
  \label{obs_incomp}
  If $\x,\y\in\partial'A$ for some stepped surface $A$,
  then $\x\not<\y$.
  This follows very naturally from the definition of a stepped surface;
  see
  the proof of Theorem~3.2 in~\cite{SELF}.
  The same
  property must hold for the set $\sigma^{-1}\partial' A$
  because the relation $<$ is preserved by $\sigma$ and its inverse $\sigma^{-1}$.

  \item
  \label{obs_expectation}
  For any $\x\in X^*$, symmetry arguments (Lemma~9.6 in~\cite{SELF}) imply that
  \[
    \E_\p[\x_A]=\x+\E_\p[\phi_A(\x)]\u/d\in \sigma\n^\perp.
  \]
  In other words, this means that  $\E_\p[\x_A]$ equals
  the unique point of intersection of the line $\x+\R\u$ with
  the hyperplane $\sigma\n^\perp$.

  \item
  \label{obs_variance}
  For any $\x\in X^*$, Lemma~10.4 in~\cite{SELF} says that
  \[
    \Var_\p\phi_A(\x)\leq \frac d 2
      \max_{B,D\in F_0(\p,\n)}\phi_B(\x)-\phi_D(\x)
      =
      \frac {\sqrt {d^3}} 2
      \max_{B,D\in F_0(\p,\n)}|\x_B-\x_D|
;
  \]
  equality is due to $\x_A=\x+\phi_A(\x)\u/d$ where $|\u/d|=1/\sqrt d$.
  Observation~\ref{obs_incomp} and the fact
  that $\sigma N\subset \partial'A$ for any $A\in F_0(\p,\n)$
  imply that
  $
  |\x_B-\x_D|
  \leq C\min_i\p_i
  $
  for any $\x\in X^*$ and $B,D\in F_0(\p,\n)$,
  where the constant $C$ depends on $\n$ only.
  We deduce that
  \[
    \Var_\p\phi_A(\x)\leq O( \min\nolimits_i\p_i)
    \]
  as $\p\to\infty$, uniformly over the choice of $\x\in X^*$.
\end{enumerate}

Let us start with an informal account of the remainder of the proof.
One should think of $\partial' A$ as
the graph of the periodic Lipschitz function $\phi_A$.
The Lipschitz property is due to the first of the previous
three observations. Periodicity allows one to
reduce this Lipschitz function to a Lipschitz function
on (appropriate discretisations of) a compact domain.
Lipschitz functions on compact domains have the convenient property
that pointwise convergence implies uniform convergence.
We are able to prove pointwise convergence
to the correct value with the help of
Observations~\ref{obs_expectation} and~\ref{obs_variance}.
The technical difficulty lies in handling rigorously the discretisation of the
boundary of the stepped surface.
Formally, we proceed as follows.

The remainder of the proof consists of two assertions.
First, we assert that
for fixed $\x\in\n^\perp$ and $\varepsilon>0$,
the random set $\sigma^{-1}\partial'A$ intersects
$B_{\x,\varepsilon}$ with high probability as $\p\to\infty$ in the measure $\mathbb P_\p$.
Here $B_{\x,\varepsilon}$ denotes an Euclidean ball of radius $\varepsilon$ centred at $\x$.
Second, we assert that $\sigma^{-1}\partial'A$ is close to
$\n^\perp$ in the Hausdorff metric if the former set intersects a finite number of suitably chosen small balls
with centres in $\n^\perp$.

Focus on the first assertion.
Fix $\mathbf x\in\n^\perp$
and $\varepsilon>0$.
It is straightforward to check
that
\[
  \sigma^{-1}(X^*+\R\u)\cap \n^\perp
  =
  \sigma^{-1}(\Z^d+\R\u)\cap\n^\perp
  =
  (\sigma^{-1}\Z^d+\R\sigma^{-1}\u)\cap\n^\perp
  \to
  \n^\perp
\]
in the Hausdorff metric as $\p\to\infty$---this is a deterministic statement.
Thus, for $\p$ sufficiently large,
we may choose $\mathbf y$ in the set on the left such that $|\x-\y|<\varepsilon/2$.
Then $\a:=P'\sigma\y\in X^*$ by choice of $\y$.
Observation~\ref{obs_expectation} states
that
\[
  \E_\p[\a_A]\in (\a+\R\u)\cap\sigma\n^\perp=\{\sigma\y\}.
\]
The vector $\a_A$ lies in $\partial'A$,
and therefore
$
  \z:=\sigma^{-1}\a_A\in\sigma^{-1}\partial'A
$.
For the first assertion it suffices to show that  $|\z-\y|<\varepsilon/2$ with high probability
as $\p\to\infty$.
The variance bound (Observation~\ref{obs_variance})
implies that $|\a_A-\sigma\y|$
is typically of order $O(\sqrt{\min_i\p_i})$
and therefore $|\z-\y|=|\sigma^{-1}\a_A-\sigma^{-1}\sigma\y|$
is typically of order $O(\sqrt{\min_i\p_i}\|\sigma^{-1}\|)=o(1)$
as $\p\to\infty$.
More precisely, $|\z-\y|<\varepsilon/2$ with high probability
as $\p\to\infty$.
%
%
%
% (Note that the choice of $\y$ depends implicitly on $\p$,
% and therefore the bound that we find must be independent of $\y$.)
% Observe that
% \[
%   |\y-\z|
%   \leq
%   \|\sigma^{-1}\|\cdot |\sigma\y-\sigma\z|
%   =
%   \|\sigma^{-1}\|\cdot |\E_\p[\a_A]-\a_A|
%   =
%   \frac{\|\sigma^{-1}\|}{\sqrt d}\cdot
%   |\mathbb E_\p [\phi_A(\a)]-\phi_A(\a)|.
% \]
% We have  $\|\sigma^{-1}\|\leq 1/\min_i\p_i$
% and $\operatorname{Var}_\p \phi_A(\a)= O(\min_i\p_i)$
% as $\p\to\infty$;
% the bound in the display on the right thus converges
% to zero in probability as $\p\to\infty$
% by the Chebyshev inequality.
This proves the first assertion.

Let us now focus on the second assertion.
It follows immediately from the first assertion
and from $N$-invariance of $\sigma^{-1}\partial'A$ that
for any $\varepsilon>0$,
we have $\mathbf n^\perp\subset\sigma^{-1}\partial'A +B_{\mathbf 0,\varepsilon}$
with high probability as $\mathbf p\to\infty$.
It is a straightforward consequence of the first
of the three observations that this also implies the converse inclusion;
that $\sigma^{-1}\partial'A\subset\mathbf n^\perp +B_{\mathbf 0,\varepsilon}$ with high probability.
Conclude that $\sigma^{-1}\partial'A\to\mathbf n^\perp$
in probability in the Hausdorff metric as $\mathbf p\to\infty$.

\end{proof}

  % !TEX root = ../ms.tex

\section{Process decomposition and martingale property}
\label{sec:process_decomposition_and_martingale_property}
We are able to asymptotically estimate $\alpha(\p,\n)$
by approximating the process $X(\0)$ by other
processes,
with vanishing errors.
First, we know that
$X(\0)$ and $\hat X+\kappa([X])$ differ by no more than some constant 
depending on $\p$ and $\n$ only,
so that the diffusivities of the two processes are equal.
Write $S^D(\p,\n)$
for the set of pairs of shapes
$(V,W)\in S(\p,\n)^2$ with the property
that $(F_1,\dots,F_{|\n|},G_1,\dots,G_{|\n|})\in D(\p,\n,2|\n|)$
where $(F_k)_k$ and $(G_k)_k$ are the fractures of $V$ and $W$ respectively.
In other words, $S^D(\p,\n)$ is the set of pairs of shapes which 
allow the application of Lemma~\ref{correction_lemma}. 
The increments of the process $\hat X+\kappa([X])$ typically behave well,
because (as will be clear later) the pair 
$([X_k],[X_{k+1}])$ is contained in 
$S^D(\p,\n)$ with high probability. 
We first introduce two auxiliary processes which facilitate the handling of the cases 
where the pair $([X_k],[X_{k+1}])$ does not belong to $S^D(\p,\n)$.
Start with the process $Y$, which is uniquely determined by the following
properties:
\begin{enumerate}
  \item  $Y_0=\hat X_0+\kappa([X_0])$,
  \item If $([X_n],[X_{n+1}])\in S^D(\p,\n)$, then  $Y_{n+1}=Y_n$,
  \item If $([X_n],[X_{n+1}])\not\in S^D(\p,\n)$, then
  \[Y_{n+1}-Y_n=\hat X_{n+1}+\kappa([X_{n+1}])-\hat X_n-\kappa([X_n]).\]
\end{enumerate}
This process should be thought of as recording the unfortunate transitions of the Markov chain.
Write $Z:=\hat X+\kappa([X])-Y$ for the difference;
this process satisfies the following properties:
\begin{enumerate}
  \item $Z_0=0$,
  \item If $([X_n],[X_{n+1}])\in S^D(\p,\n)$,
  then
  \begin{align*}
    Z_{n+1}-Z_n&=\hat X_{n+1}+\kappa([X_{n+1}])-\hat X_n-\kappa([X_n])\\
    &
    \numberthis
    \label{eq_Z_alt}
    =1_{h_1'<h_1}-1_{h_1'>h_1}+2\frac{\n_1}{|\n|}\sum_{k=1}^{|\n|}h_k'-h_k
  \end{align*}
  where $h_k$ and $h_k'$ are defined
  with respect to $f:=X_n$ and $g:=X_{n+1}$
  as in Lemma~\ref{correction_lemma},
  \item If $([X_n],[X_{n+1}])\not\in S^D(\p,\n)$,
  then $Z_{n+1}=Z_n$.
\end{enumerate}
This process should be thought of as the corrected version of the original process defined in terms of $X$.
The increments of both $Y$ and $Z$ are
bounded by $3$.
Moreover, we shall later see that either process has a well-defined diffusivity.
We prove that the diffusivity of $Y$ goes
to zero as $\p\to\infty$ because for most transitions
$([X_n],[X_{n+1}])\in S^D(\p,\n)$,
and we prove that the diffusivity of $Z$
goes to $1/(1+2|\n|)$.
This implies the main result:
that the diffusivities of the original processes $(X_n(\0))_{n}$ and $(\hat X_n)_{n}$ also tend to $1/(1+2|\n|)$.
We are able to asymptotically estimate the diffusivity of $Z$
because $Z$ is a martingale.
We shall prove this first.

\begin{theorem}
  The process $Z$ is a martingale.
\end{theorem}

\begin{proof}
  Let $f$ denote a $\p\n$-function.
  Write $\mathbb P_f$
  for the measure corresponding
  to the random walk $X$ starting from $f$.
  The goal is of course to prove that
  $\mathbb E_f(Z_1-Z_0)=0$.
  Write $\mathcal N$ for the set of shapes
  $W$ such that $W\sim[f]$ and $([f],W)\in S^D(\p,\n)$.
  The definition of $Z$ implies
  that it is sufficient
  to demonstrate that
  $\mathbb E_f((Z_1-Z_0)1_{[X_1]\in \mathcal N})=0$.
  We assume
  without loss of generality that
  $\{[X_1]\in \mathcal N\}$
  has has positive measure,
  so that our goal is to show that
  \[
  \mathbb E_f(Z_1-Z_0|[X_1]\in \mathcal N)=0.
  \]

  The distribution of $X_1$ is
  uniform in the neighbours of $f$.
  Therefore the distribution of $[X_1]$
  is uniform in the neighbours of $[f]$,
  except that the probability of obtaining
  $[X_1]=[f]$ has twice the probability of any other outcome
  (see Proposition~\ref{waarbenikmeebezig}).
  Note however that $[f]\not\in \mathcal N$
  because $([f],[f])\not\in S^D(\p,\n)$.
  Therefore it suffices to prove that
  \[
    \sum_{W\in \mathcal N}
    \mathbb E_f(Z_1-Z_0|[X_1]=W)=0.
  \]
  But now we know $[X_0]$
  and $[X_1]$ and that this pair of shapes
  is contained in $S^D(\p,\n)$,
  and therefore the difference $Z_1-Z_0$
  is given by
  \eqref{eq_Z_alt}.
  Therefore it suffices to prove that
  \[
    \sum_{W\in \mathcal N}\xi(W)=0
    \qquad\text{where}
    \qquad
    \xi(W):=1_{h_1'<h_1}-1_{h_1'>h_1}+2\frac{\n_1}{|\n|}\sum_{k=1}^{|\n|}h_k'-h_k
  \]
  and where in turn the $h_k$ are given by $f$
  and the $h_k'$ by $W$.

  We shall construct an involution
  $\tau$ on $\mathcal N$ with the property
  that it inverts the sign of $\xi$,
  that is, which satisfies $\xi(W)=-\xi(\tau(W))$
  for any $W\in \mathcal N$.
  This clearly implies the desired result.
  We shall restrict our discussion to the
  case that $|\n|=1$;
  then discuss how to generalise to
  $|\n|>1$.
  Recall that $\xi(W)$ equals
  the volume of the dark blue area
  minus the volume of the light blue
  area in the middle subfigure of
  Figure~\ref{fig_STRIPS}---if we choose
  $g$ to be the neighbour of $f$ with shape
  $W$.
  The idea is now to find an involution $\tau$ which interchanges the widths
  of the two coloured cylinders in that
  subfigure---the two coloured cylinders
  neighbouring $U_\T(G_1)$,
  where $G_1$ is the fracture characterising $W$.
  The construction which achieves
  this is pictured in Figure~\ref{fig_STRIPS_2}.
  Informally,
  the idea is to find an appropriate
  point $\m$
  around which to rotate $\partial_\T G_1$ by an angle of $\pi$
   in order to obtain a new fracture
   $\bar G_1$ with a new boundary
   $\partial_\T\bar G_1$.
   Rotating by $\pi$ twice returns the original
   fracture that we started with,
   so that $\tau$ is indeed an involution.
   We now formalise this
   construction.

   Recall that we have restricted ourself
   to studying the case
   $|\n|=1$.
   Write $F_1,G_1\in F(\p,\n)$
   for the fractures
   corresponding to the
   shapes $[f]$ and $W$
  respectively.
  Write $\bar G_1$ for
  the fracture corresponding
  to the shape $\tau(W)$,
  which we have yet to define.

Fix $A\in F_1$,
and fix $x,y\in\R$ such that
 $[x,y]:=P\partial A$.
 The topological boundary of  $U(A)$
 decomposes into two connected components,
 namely $P^{-1}(x)$ and $P^{-1}(y)$.
 As $\partial A/N$ is a compact subset of $\R^d/N$,
 it follows immediately
 that $P^{-1}(x)\cap \partial A$
 and $P^{-1}(y)\cap \partial A$
 are nonempty.
 Let $\x$ and $\y$ denote elements in these two sets respectively.
 Moreover, as $A+K$ is a union of centred unit cubes,
 it is immediate
 that the coordinates of $\x$ and $\y$
 are half-integers.
 In particular, $\x+\y\in \Z^d$.
 If $B\in G_1$,
 then it is straightforward to see that $-\partial B+\x+\y$
 is the boundary of some other $\sigma N$-invariant stepped surface
 $\bar B$.
 The desired fracture is then the fracture $\bar G_1:=\{\theta^k\bar B:k\in\Z\}$.
 Note that this is equivalent to rotating $\partial_\T G_1$ around the midpoint
 $\m$ in between $\pi\x$ and $\pi\y$ in Figure~\ref{fig_STRIPS_2}.
 This procedure interchanges the space to the left and right of $U_\T(G_1)$, and,
 as these two distances to the nearest cylinders remain positive, the new cylinder is also disjoint from
 $U_\T(F_1)$.
 The procedure is thus the involution with the desired properties, which proves the lemma
 for $|\n|=1$.
 If $|\n|>1$, then use the same procedure to
 interchange the distance on the left and right
 of each cylinder $U_\T(G_k)$
 to the nearest cylinder corresponding to a fracture of $f$.
\end{proof}

  % !TEX root = ../ms.tex

\section{Proof of the main result}
\label{sec:final_proof}

We require the following general and elementary result.

\begin{theorem}\label{minimisevariance}
Let $S$ be a finite set, $d:S\times S\to\mathbb{R}$ an antisymmetric map and $(X_n)_{n\geq 0}$
an irreducible reversible Markov chain on $S$ starting from its invariant distribution.
Define the process $A$ by $A_n:=\sum_{k=0}^{n-1}d(X_{k},X_{k+1})$.
A map $\kappa^*:S\to\R$ minimises
\begin{equation*}
  E(\kappa):=\mathbb{E}\left( \left(A_1+\kappa(X_1)-A_0-\kappa(X_0)\right)^2 \right)
%  =\mathbb{E}\left( \left(d(X_0,X_1)+\kappa(X_1)-\kappa(X_0)\right)^2 \right)
\end{equation*}
over $\kappa\in\mathbb{R}^S$
if and only if $\left(A_n+\kappa^*\left(X_n\right)\right)_{n\geq 0}$ is a martingale. Such
a map $\kappa^*$ exists and is unique up to constant differences.
Moreover, the law of $A$
converges to that of a Brownian motion
of diffusivity $E(\kappa^*)$ in the scaling limit.
\end{theorem}

The minimiser $\kappa^*$ is called the \emph{corrector} of the process $A$.

\begin{proof}[Proof of Theorem~\ref{minimisevariance}]
By writing the expectation as a finite sum
over the entries of the transition matrix of $X$
 we
see that the objective function is quadratic.
Therefore it is convex and the set of minima is an affine subspace of $\mathbb{R}^S$.
Adding a constant to the map $\kappa$ does not change $E(\kappa)$.
We also note that
$E(\kappa)\to\infty$ as the \emph{gradient} of $\kappa$
blows up, that is, whenever $\|\kappa\|\to\infty$
while keeping $\kappa(s)$ fixed for at least one $s\in S$.
Hence a minimiser of $E$ must exist and is unique up to constant differences.
Write $\kappa^*$ for such a minimiser,
so for all $s\in S$, we have
\[\frac{\partial}{\partial \kappa^*(s)}E(\kappa^*)=0.\]
By moving the derivative into the expectation
and using the detailed balance equations and antisymmetry of $d$ it is straightforward to check that
\[
\frac{\partial}{\partial \kappa(s)}E(\kappa)
=-4\mathbb{P}(X_0=s)\mathbb{E}\left(A_1+\kappa(X_1)-A_0-\kappa(X_0) \middle|X_0=s\right).
\]
 From this we conclude that, for all $s\in S$, \[\mathbb{E}\left(A_1+\kappa^*(X_1)-A_0-\kappa^*(X_0) \middle|X_0=s\right)=0.\]
 and therefore $\left(A_n+\kappa^*(X_n)\right)_{n\geq 0}$ must be a martingale.
 If $\left(A_n+\kappa^*(X_n)\right)_{n\geq 0}$ is a martingale
 then by reversing the previous argument, $\kappa^*$ is a local minimum of the objective function.
 The objective function is convex, so that $\kappa^*$ must be a global minimum.
Finally by standard arguments the distribution of $A$ converges to
that of a Brownian motion under diffusive scaling.
Write $\alpha(A)$ for the diffusivity of $A$.
 If $\left(A_n+\kappa^*(X_n)\right)_{n\geq 0}$ is a martingale then its increments are
orthogonal and identically distributed, so that
\begin{align*}
  \textstyle
  \alpha(A)&:=\lim_{n\to\infty}\frac1n\E A_n^2\\
  &\phantom{:}=\lim_{n\to\infty}\frac1n\E\left(\left(\sum\nolimits_{k=0}^{n-1}A_{k+1}+\kappa^*(X_{k+1})-A_{k}-\kappa^*(X_{k})  \right)^2\right)
  \textstyle=E(\kappa^*).
\end{align*}
The second equality follows from the fact that
$\kappa^*$ is bounded.
The penultimate expression is constant over $n$ by stationarity and orthogonality of martingale
increments, and setting $n=1$ gives the final equality.
This finishes the proof of the theorem.
\end{proof}

Now write $\P_\p$ for the measure corresponding to the random walk
$X=(X_n)_{n\geq 0}$,
and suppose the distribution of $[X]$ is invariant in $\P_\p$.
The diffusivity of the process $Y$
is bounded from above by $\E_\p(Y_1-Y_0)^2$,
and the diffusivity of the process $Z$
equals $\mathbb E_\p Z_1^2$ as $Z$ is a martingale starting from zero.
Thus, to prove the main result,
we must show that
\begin{equation}
  \label{eq:desired_limits}
\E_\p(Y_1-Y_0)^2
\leq 9\mathbb P_\p(([X_0],[X_1])\not\in S^D(\p,\n))
\to 0,
\qquad
\E_\p Z_1^2\to \frac1{1+2|\n|}
\end{equation}
as $\p\to\infty$;
for the bound on the left we use that $|Y_1-Y_0|\leq 3$.

For proving either limit,
it is sufficient to consider the distribution of the pair
$([X_0],[X_1])$ in the measure $\P_\p$.
We shall abuse notation by simply considering $\P_\p$ to be the probability measure
on $S(\p,\n)^2$ such that the distribution
of the random pair $(V,W)$ matches that of the random pair
$([X_0],[X_1])$ in the original measure.
Recall Definition~\ref{definition_of_D_nogwat}
for a description of $D(\p,\n,2|\n|)\subset F(\p,\n)^{2|\n|}$,
and recall the definition of $S^D(\p,\n)\subset S(\p,\n)^2$
at the start of Section~\ref{sec:process_decomposition_and_martingale_property}.
Write furthermore $N(\p,\n)$
for the pairs of shapes $(V,W)\subset S(\p,\n)^2$
which are neighbours.

By solving the detailed balance equation for the process $[X]$,
it is easy to see that Proposition~\ref{waarbenikmeebezig}
implies the following lemma.

\begin{lemma}
  For any $A,B\in S(\p,\n)$, we have
  \[
    \P_\p((V,W)=(A,B))
    =
    \begin{cases}
      0 &\text{if $(A,B)\not\in N(\p,\n)$,}\\
      1/T &\text{if $(A,B)\in N(\p,\n)$, but $A\neq B$,}\\
      2/T &\text{if $A=B$.}
    \end{cases}
  \]
  where $T=|N(\p,\n)|+|S(\p,\n)|$.
\end{lemma}

Let us now prove the following lemma.

\begin{lemma}
The event $S^D(\p,\n)$ has high probability in $\P_\p$
as $\p\to\infty$.
\end{lemma}

\begin{proof}
Write $\P_\p'$ for the uniform probability measure
on $N(\p,\n)$.
It suffices to demonstrate that $S^D(\p,\n)$
has high probability in $\P_\p'$.
To sample from $\P_\p'$,
one can first sample
a tuple $t:=(F_1,\dots,F_{|\n|},G_1,\dots,G_{|\n|})$
from $F(\p,\n)^{2|\n|}$ uniformly at random,
then condition on the event that
the tuples $t_f:=(F_1,\dots,F_{|\n|})$
and $t_g:=(G_1,\dots,G_{|\n|})$
define shapes which are neighbours.
Note that $t\in D(\p,\n,2|\n|)$
with high probability due to Corollary~\ref{corolor}.
Moreover, $t_f$ and $t_g$ define shapes $V$ and $W$ respectively whenever
$t\in D(\p,\n,2|\n|)$,
and, conditional on the latter event,
the probability that those shapes are indeed neighbours
is equal to
$2/{{2|\n|}\choose{|\n|}}$.
This follows from Lemma~\ref{correction_lemma},
which says that the relation $V\sim W$ depends only on the relative
ordering of the $h_k$ and $h_k'$,
and a uniformly random ordering
of these numbers satisfies the condition of that lemma with probability
$2/{{2|\n|}\choose{|\n|}}$.
Thus, conditional on the probability that the initial
proposal $t$ is accepted,
the probability that $t\in D(\p,\n,2|\n|)$,
and consequently $(V,W)\in S^D(\p,\n)$,
tends to one.
\end{proof}

\begin{proof}[Proof of Theorem~\ref{main}]
It suffices to demonstrate that the asserted
limits in \eqref{eq:desired_limits}
are correct.
The asymptotic behaviour of the limit on the left follows immediately
from the previous lemma.
For the other limit,
it now suffices to demonstrate that
\[
\E_\p(Z_1^2|S^D(\p,\n))\to \frac1{1+2|\n|}
\]
as $\p\to\infty$.
We use here that $|Z_1|\leq 3$,
so that the asymptotic expectation does not change by
conditioning on an event which has high probability.

Note that
$\E_\p(Z_1^2|S^D(\p,\n))$ equals
\[\textstyle
    \E_\p\left(
        \left(  1_{h_1'<h_1}-1_{h_1'>h_1}+2\frac{\n_1}{|\n|}\sum_{k=1}^{|\n|}h_k'-h_k
        \right)^2
    \middle|S^D(\p,\n)\right).
\]
Write $K_l$ for the event that
\[
  h_1<h_1'<\dots<h_{|\n|}<h_{|\n|}',\]
and write $K_r$ for the same event with each $h_k$ and $h_k'$ interchanged.
The distribution of
\[(h_1,\dots,h_{|\n|},h_1',\dots,h_{|\n|}')\]
in the probability measure $\P_\p(\cdot|S^D(\p,\n))$
converges weakly to the uniform continuous distribution on $[0,|\n|/\n_1)^d$
conditioned on $K_l\cup K_r$, that is, on the union of the two events in \eqref{eq:awdoiundaiwon};
this follows from similar arguments as in the proof of the previous lemma,
together with the last part Theorem~\ref{thm_asymptotic_h_r}
for the distribution of the $h_k$ and $h_k'$
and Lemma~\ref{correction_lemma} for the appropriate conditioning.

% Write $K_l$ for the event that
% $
%   h_1<h_1'<\dots<h_{|\n|}<h_{|\n|}'$,
% and write $K_r$ for the same event with each $h_k$ interchanged with $h_k'$.
Write $\P$ for the uniform measure in the set $[0,1]^{2|\n|}$
with the random variables being $h_k$ and $h_k'$.
The previous paragraph says that
\[
\textstyle
\E_\p(Z_1^2|S^D(\p,\n))\to \E\left(
\left(
1_{h_1'<h_1}-1_{h_1'>h_1}+2\sum_{k=1}^{|\n|}h_k'-h_k
\right)^2
\middle|K_l\cup K_r
\right)
\]
as $\p\to\infty$,
where the dependence on $\p$ in the measure on the left vanishes in the limit on the right so that we are truly dealing with 
a continuous model.
It suffices to demonstrate that the expression on the right in this display
is equal to $1/(1+2|\n|)$.
In fact, symmetry arguments imply that it does not matter
if we condition on $K_l$, $K_r$, or $K_l\cup K_r$.
For simplicity, we choose to condition on $K_l$.

The event $K_l$ has $\P$-probability $1/(2|\n|)!$.
Observe that
\begin{align*}
  &\textstyle\E\left(1_{K_l}
  \left(
  1_{h_1'<h_1}-1_{h_1'>h_1}+2\sum_{k=1}^{|\n|}h_k'-h_k
  \right)^2
  \right)
\\
&\quad =
\int_{[0,1]^{2|\n|}}
\textstyle
1_{K_l}
\left(
1-2\sum_{k=1}^{|\n|}h_k'-h_k
\right)^2
\\
&\quad = \int_{\Delta^{2|\n|}}
\textstyle\left(\sum_{j=0}^{2|\n|}(-1)^j x_j\right)^2dx
\\
&\quad=\frac1{(1+2|\n|)!}.
\end{align*}
For the third line, we make the following
 change of variables:
\[
  h_k =\sum_{j=0}^{2k-2}x_j,\qquad
  h_k'=\sum_{j=0}^{2k-1}x_j,\qquad
  1=\sum_{j=0}^{2|\n|}x_j.
\]
We then integrate over the unit simplex,
precisely the set where the indicator, which now depends on $x$, equals $1$.
This means that we integrate over all $x_0,x_1,\dots,x_{2|\n|}\geq 0$
which sum to $1$,
with respect to the Lebesgue measure.
For the final equality, we gather terms of equal powers and
express the integral in terms of the multivariate beta function.
Conclude that
\[
  \begin{split}
\lim_{\p\to\infty}\alpha(\p,\n)&=\lim_{\p\to\infty}\E_\p Z_1^2=
\\&=\textstyle
 \E\left(
\left(
1_{h_1'<h_1}-1_{h_1'>h_1}+2\sum_{k=1}^{|\n|}h_k'-h_k
\right)^2
\middle|K_l
\right)
\\&=
\textstyle
  \mathbb P(K_l)^{-1}
\E\left(1_{K_l}
  \left(
  1_{h_1'<h_1}-1_{h_1'>h_1}+2\sum_{k=1}^{|\n|}h_k'-h_k
  \right)^2
  \right)
\\&
=\frac{(2|\n|)!}{(1+2|\n|)!}=\frac1{1+2|\n|}
=\frac1{1+2\gcd\n},
  \end{split}
\]
which is Theorem~\ref{main}.
\end{proof}

  \section{Discussion}
\subsection{Behaviour for finite values of $\p$ and $\n$}

We had already mentioned that $\alpha(\p,\n)$ is non-increasing 
in $\p$ and $\n$ for $d=1$,
while for fixed $d\geq 2$ the diffusivity $\alpha(\p,\n)$
is not monotone in its two parameters. 
This phenomenon occurs because for $\n$ small and $\p$ large,
there is a clear distinction between the global structure of a random sample
(the macroscopic location of the fractures)
and the local structure (the microscopic deviations of the fractures from 
being flat hyperplanes).
In the limit, these structures behave independently,
and the global structure is the only structure that contributes 
to the diffusivity.
We conjecture that this separation does not occur whenever 
the entries of $\p$ and $\n$ are of the same order,
and that consequently $\alpha(\p,\n)$ is much smaller---and not typically of order $\frac1{1+2\gcd\n}$.
For example, suppose that $d\geq 2$ and that $\p=\n$,
and such that all entries are large but simultaneously $\gcd\n=1$.
In this case, there is one fracture,
but this fracture fills most of the space and touches itself in many places.
When walking on the set of height functions,
the majority of the transitions consists in changing the microscopic structure of the fracture rather 
than moving the average of the height function by a significant amount.
Thus, in this case, the picture is dominated by the microscopic structure
of the fractures, and it is not even clear if in this case one can make sense of the macroscopic 
position of the fracture.

\subsection{Failure of the martingale property for $\hat X$}

A key observation in~\cite{EGN} is the observation that the average height
$\hat X$
of the random walk $X$ is a martingale whenever $d=1$.
In that case, each fracture consists of a single down step so that
the interaction between the fractures can be controlled exactly through a combinatorial argument.
Unfortunaly, the martingale property for $\hat X$ is false for $d\geq 2$.
This can already be seen for $d=2$ for small values of $\p$ and $\n$,
by working out a small numerical example.
The property that the expected increase in average height is equal to zero,
seems to be reserved (up to perhaps some coincidence) for the case $d=1$ and
for the cases that the height function exhibits some additional symmetry
which allows for the construction of an exact involution on the entire set of neighbours.

\subsection{Relation to the six-vertex model}
Restrict now to the case $d=2$.
In that case, the gradient of each height function can be mapped to a six-vertex model with 
periodic boundary conditions around the torus.
More precisely:
the set of shapes of height functions is in bijection with the set of six-vertex configurations 
on the torus which have the correct number of up-, down-, right-, and left-arrows.
This is no surprise:
the grey up-left lines in Subfigures~\ref{fig_small_example_B} and~\ref{fig_small_example_gB}
are reminiscent of a known representation of the six-vertex model.
The uniform measure on the set of shapes can therefore be thought of 
as the square ice model on the corresponding set of six-vertex configurations.
Remark that the square ice, in turn, is related combinatorially to alternating sign matrices.
This provides yet another perspective on the uniform measure on the set of shapes. 
The invariant measure of the random walk, however,
puts equal weight on the edges of the graph whose vertices are six-vertex configurations,
and not on the vertices themselves.
It remains an open question to ask if there is a deeper relation between this stationary distribution of this random walk on the one hand,
and the six-vertex model and alternating sign matrices on the other hand.

  \renewcommand\optionalindent{}

  % FOOTER
  % !TEX root=../ms.tex
\section*{Acknowledgement}
% \addcontentsline{toc}{section}{Acknowledgement}
The author would like to thank James Norris for the inspiration to study
walks on height functions,
and for the supervision of this project.
The author would also like to thank the anonymous referee for their constructive feedback
on the manuscript.
Finally, many thanks to Fran\c{c}ois Jacopin for translating the abstract into French.

The author was supported by the Department of
Pure Mathematics and Mathematical Statistics, University of Cambridge, the UK
Engineering and Physical Sciences Research Council grant EP/L016516/1.

  \bibliographystyle{amsalpha}
  % \addcontentsline{toc}{section}{References}
  \bibliography{bib}

\end{document}